\documentclass[smallextended,referee,envcountsect,]{svjour3}
\smartqed
\usepackage{graphicx}
\usepackage{amsmath}
\usepackage{amsfonts}
\journalname{JOTA}

\begin{document}

\title{The Euler-Lagrange and Legendre  Necessary Conditions for Fractional Calculus of Variations}


\author{Shikhi Sh. Yusubov \and Shakir Sh. Yusubov \and Elimhan N. Mahmudov}

\institute{Shikhi Sh. Yusubov \at
             Department of Mathematics,
Shanghai University\\
              Shanghai, China\\
             yusubovsixi@gmail.com
           \and
            Shakir Sh. Yusubov \at
             Baku State University, Department of Mechanics and Mathematics \\
             Baku, Azerbaijan\\
              yusubov\_shakir@mail.ru
           \and
              Elimhan N. Mahmudov,  Corresponding author  \at
              Azerbaijan National Aviation Academy \\
             Baku, Azerbaijan,\\
              Azerbaijan National Academy of Sciences, Institute of Control Systems\\
               Baku, Azerbaijan,\\
Research Center for Mathematical Modeling and Optimization, UNEC\\ Baku, Azerbaijan,
              elimhan22@ yahoo.com; ORCID:0000-0003-2879-6154
}

\date{Received: date / Accepted: date}

\maketitle

\begin{abstract}
In this paper, we study the problems of minimizing  a functional
depending on the Caputo fractional derivative of order $0< \alpha
\leq 1$ and the Riemann- Liouville fractional integral of order
$\beta >0$ under certain constraints.  A fractional analogue of
the Du Bois-Reymond lemma is proved. Using this lemma for various
weak local minimum problems, the Euler-Lagrange equation is
derived in integral form. Some serious works in the literature claim that the standard proof of the Legendre condition in the classical case $\alpha=1$ cannot be adapted to the fractional case $0<\alpha <1$ with final constraints. In spite of this, we prove the Legendre conditions using the standard classical method. The obtained
necessary conditions are illustrated by appropriate examples. 
\end{abstract}


\section{Introduction}

The main problem of the calculus of variations arose as a direct
generalization of the brachistochrone problem posed by I. Bernulli
in 1696. This problem contained typical features of a new class of
mathematical problems and throughout the history of the calculus
of variations served as both an object of testing new methods and
a basis for many interesting and important generalizations. For
the first time in 1742, Euler, approximating curves with broken
lines, derived the necessary conditions in the form of a
second-order differential equation, which the extremals had to
satisfy. Lagrange later called this equation the Euler equation. Lagrange
himself derived this equation in 1755 by varying the curve that
was supposed to be the extremum. The method of variations proposed
by Lagrange is the main method of theoretical study of
extremal problems in functional spaces. In this method, the type
of variation used plays a decisive role, since both the simplicity
and the strength of the resulting necessary condition largely
depend on it. When studying variational problems, the space of
functions on which the functional is considered plays an important
role (see, e.g. \cite{15,20,31}).

Since the end of the 20th century, the emergence of applied
problems related to fractional integrals and derivatives in
various fields (see, e.g., \cite{12,21,22}) has led to
intensive study of problems expressed by fractional integrals and
derivatives (see, e.g. \cite{23,28,29,30}). One such field is
the calculus of variations (various necessary conditions for optimality of the first and higher order are obtained for various systems \cite{1n,32n,22n,23n,24n,25n,33n,35n}). The article \cite{21n} considers an optimal control problem controlled by nonlinear fractional-order systems with multiple time-variable delays and subject to canonical constraints, where fractional-order derivatives are expressed in the Caputo sense.

The first question that arises when we want to obtain results
similar to those obtained in the classical calculus of variations
in the fractional calculus of variations is: what should be the
integral functional whose extremum is sought, and the space of
functions in which this functional is defined?

It is known that the integral functional studied in the classical
calculus of variations is a generalization of the integral
functional in the brachistochrone problem, coming directly from
practice. In fractional variational calculus, as for as we know,
the integral functional that follows from practice is not yet
known. Therefore, in the early years of fractional variational
calculus, the integral functional was considered in the same way
as in classical variational calculus (see, e.g.,
\cite{2,4,5,6,7,9,24}). It should be noted that the fractional
Euler-Lagrange equation was first formulated in the literature in
paper \cite{27}, where it was shown that nonconservative forces are modeled by noninteger derivatives, which, due to their non-local nature, can better describe the behavior of some natural processes and attracted the attention of not only mathematicians, but also physicists, engineers, and others. In the paper \cite{5} two variational problems of finding the Euler–Lagrange equations corresponding to Lagrangians containing fractional derivatives of real- and complex-order are considered. The first one is the unconstrained variational problem, while the second one is the fractional optimal control problem.

In 2019, it was shown in \cite{14} that to ensure the existence of a
solution in fractional variational calculus, it is necessary to
take into account the fractional Riemann-Liouville integral. In
the integral functional for which the extremum is sought is not a
Riemann-Liouville fractional integral, then the problem under
consideration may not have a solution. In this paper, the integral
functional is considered in the form of a fractional
Riemann-Liouville integral of order $0< \alpha \leq 1$, and the
Euler-Lagrange equation is obtained in integral and differential
form. In paper \cite{10}, a more general problem is considered, namely
the problem of minimizing the general Bolza functional, which
contains the Mayer cost, as well as the Lagrange cost, written as
a fractional Riemann-Liouville integral of order $\beta >0$ and
depending on the fractional Caputo derivative of order $0<\alpha
\leq 1$. Further, the necessary first-order optimality
condition (Euler-Lagrange differential equation with
Riemann-Liouville fractional  derivative), transversality
conditions and Legendre condition (second order) (see  Theorem
3.1 \cite{10}) are obtained for the case when the initial and final
conditions  are not fixed. But the relationship between the
parameters $0< \alpha \leq 1$ and $\beta >0$ has not been studied.
It is said there that the necessary first order optimality
condition (Euler-Lagrange equation) and transversality conditions
can be obtained similarly under the following special cases:

(i) fixed initial and free final condition;

(ii) free initial and fixed final condition;

(iii) fixed initial and fixed final condition.

In their opinion, the standard proof of the Legendre condition (second order) in the
classical case $\alpha=1$ cannot be adapted to the fractional case $0< \alpha < 1$ when dealing with final constraints. They explain this by the fact that there is no non-trivial function
which, together with its fractional Caputo derivative of order $0< \alpha < 1$, would have
compact support. Since the proof of the Legendre condition in the classical case $\alpha=1$
uses a non-trivial variation which, together with its usual derivative, has compact support.
The authors of \cite{10} called this fact an obstacle to the proof of the Legendre condition in the presence of finite constraints and concluded that the method used in the classical theory cannot be extended to the fractional case and having overcome this difficulty, found an alternative proof of the Legendre condition in the classical case $\alpha=\beta=1$. They adapted the well-known proof of Pontryagin’s maximum principle \cite{8}, based on Ekland’s variational principle \cite{13}, to a fractional problem of the calculus of variations with general mixed initial-final constraints. The Euler-Lagrange equation, transversality conditions and the Legendre condition are obtained (see \cite{10} Theorem 3.2). L. Bourdin’s doctoral dissertation \cite{11} argued that a more natural fractional problem, which avoids the trap of studying a fractional calculus of variations problem, is to take $\alpha=\beta$.

While reading the works \cite{10,11}, we had several questions. First, we had
doubts about whether it is really impossible to obtain the
Legendre necessary condition using the methods of classical
calculus of variations when dealing with final constraints, and
second, whether it is possible to investigate the relationship
between the parameters $\alpha$ and $\beta$ when obtaining the
necessary optimality conditions. However, so far we have not been
able to find a work in the literature that would adapt the
standard proof of the Legendre condition in the classical case
$\alpha=1$ to the fractional case $0<\alpha<1$ when dealing with
final constraints. So we really wanted to fill that gap and do
some more in-depth research into these types of issues. It should
be noted the starting point for this work was works \cite{10,11}.

In the present paper, we consider a fractional variational problem for
which the functional is written as a fractional Riemann-Liouville
integral of order $\beta >0$ and depends on the fractional Caputo
derivative of order $0<\alpha \leq 1.$ We are looking for
solutions to the stated fractional variational problem in the
class of functions for which the functions together with the
fractional Caputo derivatives of order $0 < \alpha \leq 1$ are
continuous (or in the class of functions for which the functions
are continuous and the fractional Caputo derivatives of order $0<
\alpha \leq 1$ are piecewise continuous) on a given interval. It
should be noted that the functional and the space of functions are
a direct generalization of the classical case when
$\alpha=\beta=1.$ As in the classical case, we first prove a lemma
that is a direct generalization of the classical Du Bois-Reymond
lemma. It is shown here that the relationship between the
parameters $\beta$ and $\alpha$ must be taken into account. This
lemma for $\beta=\alpha=1$ agrees with then classical Du Bois-Reymond lemma, and for $0<\beta=\alpha \leq 1$ agrees with the
lemma obtained in article \cite{14}. Further, using this lemma for the
cases considered above ((i), (ii), (iii) and both ends are free),
the necessary Euler-Lagrange conditions for a weak local minimum
are proved, for which, in contrast to the Euler-Lagrange condition
obtained in work \cite{10}, the connection between the parameters
$\beta$ and $\alpha$ is directly visible. The importance of this
relationship is illustrated by the relevant example in Remark 4.1.
Corresponding illustrative examples are provided. Since in the
classical calculus of variations the type of variation used plays
a decisive role, it turns out that in the fractional calculus of
variations the type of variation used also plays an important
role. We, by introducing a special variation depending on the
parameter $k$,in contrast to paper \cite{10}, show that the standard
proof of the Legendre condition in the classical case
$\alpha=1$ can be adapted to the fractional case $0< \alpha <1$
when dealing with final constraints. In cases (ii) and (iii) we
have proved the Legendre conditions, and in the other two cases
the Legendre conditions were proved in papers \cite{10,11,18,33,32}.

It is known that the modern approach to deriving the classical
Euler-Lagrange equation essentially consists of using the formula
of integration by parts and the Du Bois-Reymond lemma. In our
approach we do not use integration by parts, but only with the
help of the generalized Du Bois-Reymond lemma we obtain the
Euler-Lagrange equations for $0< \alpha \leq 1$ and $\beta >0.$
Remark 4.2 shows that for $\beta=\alpha=1$ the necessary condition
we have obtained differs in form from the classical one, but from
this condition we immediately obtain the classical case. From
Theorems \ref{teo45} and \ref{theo46} it follows that in problems where the right
ends are free, then for the existence of solutions of such
problems for $\beta> \alpha>0$ along the extremal it is necessary
that $l_{x_{1}}=0.$\par
Thus, the main novelty of this work lies in one of the possible
approaches to proving the first and second order optimality
criteria for a fractional variational problem using the
generalized Du Bois-Reymond lemma, namely the Euler-Lagrange
equation and the Legendre condition for a weak minimum.

The rest of the paper is organized as follows. 

In Section 2, the
definitions and basic properties of fractional order integrals and
derivatives are recalled, and also some preliminary results. 

In
Section 3, the fundamental Du Bois-Reymond lemma for fractional
variational calculus is proved and the continuity of a certain
function that plays an important role in the study Euler-Lagrange
equation is proved.

Section 4 is devoted to the proof of the necessary Euler-Lagrange
and Legendre conditions for four cases. Corresponding illustrative
examples are given.
\section{Notations  and basics from fractional calculus}\label{sec1}
This section is devoted to recall basic definitions and results
about Riemann-Liouville and Caputo fractional operators. All of
the presented below is very standard and mostly extracted from the
monographs \cite{23,29}.

Let $\mathbb{R}^{n}$ and $\mathbb{R}^{n\times n}$ be the spaces of
$n$- dimensional vectors and $(n \times n)$- matrices. By
$\|\cdot\|$, we denote a norm in $\mathbb{R}^{n}$ and the
corresponding norm in $\mathbb{R}^{n \times n}$. Let numbers
$t_{0}, \ t_{1} \in \mathbb{R}$, $t_{0}<t_{1}$ be fixed, and let
$X$ be one of the spaces $\mathbb{R}^{n }$ or $\mathbb{R}^{n
\times n}$. By $\langle a, \ b \rangle$ we denote the scalar
product of vectors $a$ and $b$.

Let $C([t_{0}, t_{1}], X)$ the space of continuous functions on
$[t_{0}, t_{1}]$ with values in $X$, endowed with the uniform norm
$\|\cdot\|_{c}.$  $L^{1}([t_{0}, t_{1}], X)$ the Lebesgue space of
summable functions defined on $[t_{0}, t_{1}]$ with values in $X$,
endowd with its usual norm $\|\cdot\|_{L^{1}}.$
$L^{\infty}([t_{0}, t_{1}], X)$ the Lebesgue space of essentially
bounded functions defined on $[t_{0}, t_{1}]$ with values in $X$,
endowed with its usual norm $\|\cdot\|_{L^{\infty}}.$

\begin{definition}\label{def21} Let $\alpha>0.$ For a function
$\varphi:[t_{0}, t_{1}]\rightarrow X,$ the left-sided and
right-sided Riemann-Liouville fractional integrals of the order
$\alpha$ are defined for $t \in [t_{0}, t_{1}]$ by
$$
(I_{t_{0}+}^{\alpha}\varphi)(t)=\frac{1}{\Gamma(\alpha)}\int\limits_{t_{0}}^{t}(t-\tau)^{\alpha-1}\varphi(\tau)d\tau\,\,\,
\texttt{and} \,\,\,
(I_{t_{1}-}^{\alpha}\varphi)(t)=\frac{1}{\Gamma(\alpha)}\int\limits_{t}^{t_{1}}(\tau-t)^{\alpha-1}\varphi(\tau)d\tau,
$$
respectively. For $\alpha=0$ we set $(I_{t_{0}+}^{\alpha}\varphi):=\varphi$ and $(I_{t_{1}-}^{\alpha}\varphi):=\varphi$.

If $\varphi \in L^{\infty}([t_{0}, t_{1}], X),$ then the above
functions are defined and finite everywhere on $[t_{0}, t_{1}]$.
\end{definition}
\begin{definition}\label{def22}Let $\alpha \in [0, \ 1],$ and $\varphi
\in L^{1}([t_{0}, t_{1}], X).$ For a function $\varphi$ the
left-sided and right-sided Riemann-Liouville fractional
derivatives of the order $\alpha$ are defined for $t \in [t_{0},
t_{1}]$ by
$$
(D_{t_{0}+}^{\alpha}\varphi)(t)=\frac{d}{dt}(I_{t_{0}+}^{1-\alpha}\varphi)(t)=
\frac{1}{\Gamma(1-\alpha)}\frac{d}{dt}
\int\limits_{t_{0}}^{t}(t-\tau)^{-\alpha}\varphi(\tau)d\tau\,
$$
and
$$
(D_{t_{1}-}^{\alpha}\varphi)(t)=-\frac{d}{dt}(I_{t_{1}-}^{1-\alpha}\varphi)(t)=
-\frac{1}{\Gamma(1-\alpha)}\frac{d}{dt}
\int\limits_{t}^{t_{1}}(\tau-t)^{-\alpha}\varphi(\tau)d\tau,
$$
if $(I_{t_{0}+}^{1-\alpha}\varphi)(\cdot)$ and
$(I_{t_{1}-}^{1-\alpha}\varphi)(\cdot)$  are absolutely
continuous functions on $[t_{0}, t_{1}]$.
\end{definition}
\begin{definition}\label{def23} Let $\alpha \in [0, \ 1],$ and $\varphi
\in C([t_{0}, t_{1}], X).$ For a function $\varphi$ the left-sided
and right-sided Caputo  fractional derivatives of the order
$\alpha$ are defined for $t \in [t_{0}, t_{1}]$ by
$$
(^{c}D_{t_{0}+}^{\alpha}\varphi)(t)=\frac{d}{dt}(I_{t_{0}+}^{1-\alpha}(\varphi(\cdot)-\varphi(t_{0})))(t)=
\frac{1}{\Gamma(1-\alpha)}\frac{d}{dt}
\int\limits_{t_{0}}^{t}(t-\tau)^{-\alpha}(\varphi(\tau)-\varphi(t_{0}))d\tau\,
$$
and
$$
(^{c}D_{t_{1}-}^{\alpha}\varphi)(t)=-\frac{d}{dt}(I_{t_{1}-}^{1-\alpha}(\varphi(\cdot)-\varphi(t_{1})))(t)
$$
$$
=-\frac{1}{\Gamma(1-\alpha)}\frac{d}{dt}
\int\limits_{t}^{t_{1}}(\tau-t)^{-\alpha}(\varphi(\tau)-\varphi(t_{1}))d\tau,
$$
if $(I_{t_{0}+}^{1-\alpha}(\varphi(\cdot)-\varphi(t_{0})))(\cdot)$
and
$(I_{t_{1}-}^{1-\alpha}(\varphi(\cdot)-\varphi(t_{1})))(\cdot)$
are absolutely continuous functions on $[t_{0}, t_{1}]$.

By $C^{\alpha}([t_{0}, t_{1}], X),$ $0<\alpha \leq 1,$ we denote
the set of all functions $\varphi:[t_{0}, t_{1}]\rightarrow X,$
such that
$$
\varphi(t)=\varphi(t_{0})+(I_{t_{0}+}^{\alpha}\psi)(t), \,\,\, t
\in [t_{0}, t_{1}],
$$
with $\psi \in C([t_{0}, t_{1}], X).$

$$
C_{0}^{\alpha}([t_{0}, t_{1}], X)=\{\varphi \in C^{\alpha}([t_{0},
t_{1}], X| \varphi(t_{0})=\varphi(t_{1})=0\}.
$$
\end{definition}
\begin{proposition}\label{pro21}
For any $\varphi \in C^{\alpha}([t_{0},
t_{1}], X),$ $0<\alpha \leq 1,$ the value
$(^{c}D_{t_{0}+}^{\alpha}\varphi)(t)$ is correctly defined for
every $t \in [a, \ b].$ Moreover, the inclusion
$(^{c}D_{t_{0}+}^{\alpha}\varphi)$ $\in C([t_{0}, t_{1}], X)$ holds
(i.e., there exists $\psi \in C([t_{0}, t_{1}], X)$ such that
$\psi(t)$ $=(^{c}D_{t_{0}}^{\alpha}\varphi)(t)$ for every $t \in
[t_{0}, t_{1}]$) and
$$
(I_{t_{0}+}^{\alpha}(^{c}D_{t_{0}+}^{\alpha}\varphi))(t)=\varphi(t)-\varphi(t_{0}),
\,\,\, t \in [t_{0}, t_{1}].
$$
\end{proposition}
\begin{proposition}\label{pro22}Let $\alpha>0$ and $\varphi \in
C([t_{0}, t_{1}], X),$ then
$(^{c}D_{t_{0}+}^{\alpha}(I_{t_{0}+}^{\alpha}\varphi))(t)=\varphi(t)$
and
$(^{c}D_{t_{1}-}^{\alpha}(I_{t_{1}-}^{\alpha}\varphi))(t)=\varphi(t)$
for every $t \in [t_{0}, t_{1}].$

By $PC^{\alpha}([t_{0}, t_{1}], \mathbb{R}^{n})$ we denote the set
of functions that are continuous, and their fractional Caputo
derivatives of order $\alpha$ are continuous everywhere on
$[t_{0}, t_{1}],$ except for are finite number of points
$\sigma_{i}\in (t_{0}, t_{1}),$ $i=\overline{1, \ k}.$ In this
case, at the points $\sigma_{i},$ the fractional derivative
$^{c}D_{t_{0}+}^{\alpha}x$ has discontinuities of the first kind.

$$
PC_{0}^{\alpha}([t_{0}, t_{1}], \mathbb{R}^{n})=\{\varphi \in
PC_{0}^{\alpha}([t_{0}, t_{1}],
\mathbb{R}^{n})|\varphi(t_{0})=\varphi(t_{1})=0\}.
$$
\end{proposition}
\begin{lemma}\label{lem21}([Lemma 3.3]\cite{10})  Let $0<\alpha \leq 1.$ The
inequality
$$
0 \leq (\sigma_{2}^{\alpha}-\sigma_{1}^{\alpha})^{2}\leq
\alpha(\sigma_{2}-\sigma_{1})^{\alpha+1}\sigma_{1}^{\alpha-1},
$$
holds true for all $0 < \sigma_{1}\leq \sigma_{2}.$
\end{lemma}
\section{The fundamental lemma of fractional calculus of variations}\label{sec3}
In this section we prove the generalized Du Bois-Reymond lemma,
which plays an essential role in deriving the first-order
necessary condition.
\begin{lemma}\label{lem31}
 Suppose that $f \in C([t_{0},
t_{1}],\mathbb{R}), \,$ $0<\alpha \leq 1$ and $\beta > 0.$ Then to
fulfill equality
\begin{equation}\label{equ1}
\int\limits_{t_{0}}^{t_{1}}(t_{1}-t)^{\beta-1}f(t)(^{c}D_{t_{0}+}^{\alpha}h)(t)dt=0,
\,\,\, \forall h \in C_{0}^{\alpha}([t_{0}, t_{1}], \mathbb{R})
\end{equation}%
it is necessary and sufficient that

1) $(t_{1}-t)^{\beta-\alpha}f(t)=0$ on $[t_{0}, t_{1}],$ if $\beta
> \alpha>0,$

2) $f(t)=\frac{k}{\Gamma(\alpha)}(t_{1}-t)^{\alpha-\beta}$ on
$[t_{0}, t_{1}],$ if $0 < \beta \leq \alpha \leq1$ for some
constant $k\neq 0.$
\end{lemma}
\begin{proof} {\bf Sufficiency}. From $f(t)=0,$ $t \in [t_{0}, t_{1}]$
the equality (\ref{equ1}) immediately follows. Let now conditions 2) be
satisfied. Then
$$
\int\limits_{t_{0}}^{t_{1}}(t_{1}-t)^{\beta-1}f(t)(^{c}D_{t_{0}+}^{\alpha}h)(t)dt=
\frac{k}{\Gamma(\alpha)}\int\limits_{t_{0}}^{t_{1}}(t_{1}-t)^{\beta-1}(t_{1}-t)^{\alpha-\beta}(^{c}D_{t_{0}+}^{\alpha}h)(t)dt
$$
$$
=\frac{k}{\Gamma(\alpha)}\int\limits_{t_{0}}^{t_{1}}(t_{1}-t)^{\alpha-1}(^{c}D_{t_{0}+}^{\alpha}h)(t)dt=k(h(t_{1})-h(t_{0}))=0.
$$

\textbf{Necessity.} Let $f \in C([t_{0}, \ t_{1}], \mathbb{R}),$
$0 < \alpha \leq 1,$ $\beta>0$ and equality (\ref{equ1}) is satisfied.
First consider the case $\beta > \alpha$. Let's put
$$
h(t)=\frac{1}{\Gamma(\alpha)}\int\limits_{t_{0}}^{t}(t-\tau)^{\alpha-1}
[\Gamma(\alpha)(t_{1}-\tau)^{\beta-\alpha}f(\tau)-k_{0}]d\tau,
\,\, t \in [t_{0}, \ t_{1}],
$$
where $k_{0}$ is a number.

The fractional Caputo derivative of the function $h$ has the form
$$
(^{c}D_{t_{0}+}^{\alpha}h)(t)=\Gamma(\alpha)(t_{1}-t)^{\beta-\alpha}f(t)-k_{0},
\,\, t \in [t_{0}, \ t_{1}].
$$

Therefore, the function $h$ belongs to the set $C^{\alpha}([t_{0},
\ t_{1}], \mathbb{R})$ and $h(t_{0})=0.$ If you choose $k_{0}$ in
the form:
$$
k_{0}=\frac{\Gamma(\alpha+1)}{(t_{1}-t_{0})^{\alpha}}\int\limits_{t_{0}}^{t_{1}}(t_{1}-t)^{\beta-1}f(t)dt,
$$
then it is obvious that
$$
h(t_{1})=\frac{1}{\Gamma(\alpha)}\int\limits_{t_{0}}^{t_{1}}(t_{1}-\tau)^{\alpha-1}
[\Gamma(\alpha)(t_{1}-\tau)^{\beta-\alpha}f(\tau)-k_{0}]d\tau=0.
$$

Thus, for a given $k_{0}$, the function $h$ belongs to the set
$C_{0}^{\alpha}([t_{0},t_{1}], \mathbb{R}).$ For the chosen $h$ we
write equality (\ref{equ1}) in the form
$$
\frac{1}{\Gamma(\alpha)}\int\limits_{t_{0}}^{t_{1}}(t_{1}-t)^{\alpha-1}
\Gamma(\alpha)(t_{1}-t)^{\beta-\alpha}f(t)[\Gamma(\alpha)(t_{1}-t)^{\beta-\alpha}f(t)
-k_{0}]dt=0.
$$

On the other hand
$$
\frac{k_{0}}{\Gamma(\alpha)}\int\limits_{t_{0}}^{t_{1}}(t_{1}-t)^{\alpha-1}
[\Gamma(\alpha)(t_{1}-t)^{\beta-\alpha}f(t) -k_{0}]dt=0.
$$

From the last two equalities we have
$$
\frac{1}{\Gamma(\alpha)}\int\limits_{t_{0}}^{t_{1}}(t_{1}-t)^{\alpha-1}
[\Gamma(\alpha)(t_{1}-t)^{\beta-\alpha}f(t) -k_{0}]^{2}dt=0.
$$

It follows that
$$
\Gamma(\alpha)(t_{1}-t)^{\beta-\alpha}f(t)=k_{0}, \,\,\, t \in
[t_{0}, \ t_{1}].
$$

Taking into account the condition $\beta> \alpha$, from this
equality for $t=t_{1}$ we obtain that $k_{0}=0$. Then for any $t
\in [t_{0}, \ t_{1}]$ we have $(t_{1}-t)^{\beta-\alpha}f(t)=0.$

Now consider the case $0 < \beta \leq \alpha \leq 1.$ For $h \in
C_{0}^{\alpha}([t_{0},{t_1}], \mathbb{R})$ we have $h(t_{0})=0$
and
\begin{equation}\label{equ2}
h(t_{1})=\frac{1}{\Gamma(\alpha)}\int\limits_{t_{0}}^{t_{1}}(t_{1}-t)^{\alpha-1}
(^{c}D^{\alpha}_{t_{0}+}h)(t)dt=0. 
\end{equation}%

Multiplying equality (\ref{equ2}) by $0 \neq k \in \mathbb{R}$ and
subtracting from equality (\ref{equ1}), we get
\begin{equation}\label{equ3}
\int\limits_{t_{0}}^{t_{1}}(t_{1}-t)^{\beta-1}\left[f(t)-\frac{k}{\Gamma(\alpha)}(t_{1}-t)^{\alpha-\beta}
\right](^{c}D^{\alpha}_{t_{0}+}h)(t)dt=0. 
\end{equation}%

Now we choose a constant number $k$ so that equation
\begin{equation}\label{equ4}
(^{c}D^{\alpha}_{t_{0}+}h)(t)=f(t)-\frac{k}{\Gamma(\alpha)}(t_{1}-t)^{\alpha-\beta},
\,\, t \in [t_{0}, t_{1}], 
\end{equation}%
in the space $C_{0}^{\alpha}([t_{0}, t_{1}], \mathbb{R})$ has a
unique solution.

Obviously, the solution to equation (\ref{equ4}), satisfying the condition
$h(t_{0})=0$, has the form
$$
h(t)=\frac{1}{\Gamma(\alpha)}\int\limits_{t_{0}}^{t}(t-\tau)^{\alpha-1}\left(f(\tau)-\frac{k}{\Gamma(\alpha)}(t_{1}-\tau)^{\alpha-\beta}
\right)d\tau.
$$

From the conditions
$$
h(t_{1})=\frac{1}{\Gamma(\alpha)}\int\limits_{t_{0}}^{t_{1}}(t_{1}-\tau)^{\alpha-1}\left(f(\tau)-\frac{k}{\Gamma(\alpha)}(t_{1}-\tau)^{\alpha-\beta}
\right)d\tau=0,
$$
we find
$$
k=\frac{(2\alpha-\beta)\Gamma(\alpha)}{(t_{1}-t_{0})^{2\alpha-\beta}}\int\limits_{t_{0}}^{t_{1}}(t_{1}-t)^{\alpha-1}f(t)dt.
$$

Taking into account (\ref{equ4}) of (\ref{equ3}) we have
$$
\int\limits_{t_{0}}^{t_{1}}(t_{1}-t)^{\beta-1}\left[f(t)-\frac{k}{\Gamma(\alpha)}(t_{1}-t)^{\alpha-\beta}
\right]^{2}dt=0.
$$
If follows that
$$
\qquad\qquad f(t)=\frac{k}{\Gamma(\alpha)}(t_{1}-t)^{\alpha-\beta}, \,\,\, t
\in [t_{0}, t_{1}], \;\mbox{if} \;\,\, 0< \beta \leq \alpha \leq 1.\qquad\qquad\quad\square
$$
\end{proof}

Now let's introduce the next function, which is very important for us in further research
$$
(Sa)(t)=\frac{(t_{1}-t)^{1-\beta}}{\Gamma(\alpha)}\int\limits_{t}^{t_{1}}(t_{1}-\tau)^{\beta-1}(\tau-t)^{\alpha-1}a(\tau)d\tau,
\,\,\, t \in [t_{0}, t_{1}],
$$
where $0< \alpha \leq 1$, $\beta >0$.

The following lemma is true.
\begin{lemma}\label{lem32}
Let $0 < \alpha \leq 1$ and $\beta > 0$ real
numbers. If $a \in C([t_{0}, t_{1}],\mathbb{R}),$ then $Sa \in
C([t_{0}, \ t_{1}], \mathbb{R}).$
\end{lemma}
\begin{proof} First consider the case $\beta> \alpha$. Let
$M=\max\limits_{t \in [t_{0}, t_{1}]}|a(t)|$ and $t_{0} \leq
s_{1}<s_{2}< t_{1}$. Then we have
$$
(Sa)(s_{2})-(Sa)(s_{1})=\frac{1}{\Gamma(\alpha)}\int\limits_{s_{2}}^{t_{1}}(t_{1}-\tau)^{\beta-1}[(t_{1}-s_{2})^{1-\beta}
(\tau-s_{2})^{\alpha-1}
$$
$$
-(t_{1}-s_{1})^{1-\beta}(\tau-s_{1})^{\alpha-1}]a(\tau)d\tau-\frac{(t_{1}-s_{1})^{1-\beta}}{\Gamma(\alpha)}\int\limits_{s_{1}}^{s_{2}}(t_{1}-\tau)^{\beta-1}(\tau-s_{1})^{\alpha-1}
a(\tau)d\tau.
$$

It is obvious that for $\beta > \alpha$ the function
$l(s)=(t_{1}-s)^{1-\beta}(\tau-s)^{\alpha-1},$ $t_{0}< s< \tau
\leq t_{1},$ increases monotonically. Therefore
$$
(t_{1}-s_{2})^{1-\beta}(\tau-s_{2})^{\alpha-1}-(t_{1}-s_{1})^{1-\beta}(\tau-s_{1})^{\alpha-1}\geq
0, \,\,\, t_{0}\leq s_{1}< s_{2}< \tau \leq t_{1}.
$$
Then we get
$$
|(Sa)(s_{2})-(Sa)(s_{1})|\leq M[(S1)(s_{2})-(S1)(s_{1})]
$$
$$
+2M\frac{(t_{1}-s_{1})^{1-\beta}}{\Gamma(\alpha)}\int\limits_{s_{1}}^{s_{2}}(t_{1}-\tau)^{\beta-1}(\tau-s_{1})^{\alpha-1}
d\tau
$$
$$
=\frac{M\Gamma(\beta)}{\Gamma(\alpha+\beta)}[(t_{1}-s_{2})^{\alpha}-(t_{1}-s_{1})^{\alpha}]
+2M\frac{(t_{1}-s_{1})^{1-\beta}}{\Gamma(\alpha)} \int
\limits_{s_{1}}^{s_{2}}(t_{1}-\tau)^{\beta-1}(\tau-s_{1})^{\alpha-1}d\tau
$$
$$
\leq 2M\frac{(t_{1}-s_{1})^{1-\beta}}{\Gamma(\alpha)}\int
\limits_{s_{1}}^{s_{2}}(t_{1}-\tau)^{\beta-1}(\tau-s_{1})^{\alpha-1}d\tau.
$$

Consider the following cases:

(i) If $\alpha< \beta \leq 1,$ then
$$
|(Sa)(s_{2})-(Sa)(s_{1})|\leq
\frac{2M}{\Gamma(\alpha+1)}(t_{1}-s_{1})^{1-\beta}(t_{1}-s_{2})^{\beta-1}
(s_{2}-s_{1})^{\alpha}.
$$

(ii) If $\beta>1$, then
$$
|(Sa)(s_{2})-(Sa)(s_{1})|\leq \frac{2M}{\Gamma(\alpha+1)}
(s_{2}-s_{1})^{\alpha}.
$$

Therefore the function $Sa$ is continuous on $[t_{0}, t_{1}).$ On
the other hand, for any $t \in [t_{0}, t_{1})$ we have
$$
|(Sa)(t)|\leq
\frac{M(t_{1}-t)^{1-\beta}}{\Gamma(\alpha)}\int\limits_{t}^{t_{1}}(t_{1}-\tau)^{\beta-1}
(\tau-t)^{\alpha-1}d\tau=\frac{M\Gamma(\beta)}{\Gamma(\alpha+\beta)}
(t_{1}-t)^{\alpha}.
$$

Thus $\lim\limits_{t\rightarrow t_{1}}(Sa)(t)=0.$ Consequently,
$(Sa)(t)$ can be continuously extended by $0$ in $t=t_{1}$.
Finally, if $\beta> \alpha,$ then for any $a \in C ([t_{0},
t_{1}], \mathbb{R})$ we have $Sa \in C([t_{0}, t_{1}],
\mathbb{R}).$

Now consider the case $0< \beta \leq \alpha \leq 1.$ Let's denote
$$
(Sa)(t)=(t_{1}-t)^{1-\beta}\psi(t),
$$
where
$$
\psi(t)=\frac{1}{\Gamma(\alpha)}\int\limits_{t}^{t_{1}}(t_{1}-\tau)^{\beta-1}(\tau-t)^{\alpha-1}a(\tau)d\tau.
$$

Then for any $s_{1},\ s_{2}$ $(t_{0}\leq s_{1}<s_{2}<t_{1})$, we
have
$$
\psi(s_{2})-\psi(s_{1})=\frac{1}{\Gamma(\alpha)}\int\limits_{s_{2}}^{t_{1}}
(t_{1}-\tau)^{\beta-1}((\tau-s_{2})^{\alpha-1}-(\tau-s_{1})^{\alpha-1})a(\tau)d\tau
$$
$$
-\frac{1}{\Gamma(\alpha)}\int\limits_{s_{1}}^{s_{2}}
(t_{1}-\tau)^{\beta-1}(\tau-s_{1})^{\alpha-1}a(\tau)d\tau.
$$

It is obvious that
$(\tau-s_{2})^{\alpha-1}-(\tau-s_{1})^{\alpha-1}\geq 0,$
$t_{0}\leq s_{1}<s_{2}< \tau \leq t_{1}.$ Therefore
$$
|\psi(s_{2})-\psi(s_{1})|\leq
\frac{M}{\Gamma(\alpha)}\int\limits_{s_{2}}^{t_{1}}(t_{1}-\tau)^{\beta-1}(\tau-s_{2})^{\alpha-1}d\tau-
\frac{M}{\Gamma(\alpha)}\int\limits_{s_{1}}^{t_{1}}(t_{1}-\tau)^{\beta-1}(\tau-s_{1})^{\alpha-1}d\tau
$$
$$
+\frac{2M}{\Gamma(\alpha)}\int\limits_{s_{1}}^{s_{2}}(t_{1}-\tau)^{\beta-1}(\tau-s_{1})^{\alpha-1}d\tau
\leq
\frac{M\Gamma(\beta)}{\Gamma(\alpha+\beta)}((t_{1}-s_{2})^{\alpha+\beta-1}-(t_{1}-s_{1})^{\alpha+\beta-1})
$$
$$
+\frac{2M}{\Gamma(\alpha+1)}(t_{1}-s_{2})^{\beta-1}(s_{2}-s_{1})^{\alpha}.
$$

If $\alpha+\beta-1 \geq 0,$ then
$(t_{1}-s_{2})^{\alpha+\beta-1}-(t_{1}-s_{1})^{\alpha+\beta-1}\leq
0.$ In this case
\begin{equation}\label{equ5}
|\psi(s_{2})-\psi(s_{1})|\leq
\frac{2M}{\Gamma(\alpha+1)}(t_{1}-s_{2})^{\beta-1}(s_{2}-s_{1})^{\alpha}.
\end{equation}%

If $\alpha+\beta-1<0,$ then
$$
(t_{1}-s_{2})^{\alpha+\beta-1}-(t_{1}-s_{1})^{\alpha+\beta-1}\leq
(t_{1}-s_{2})^{\alpha+\beta-1}(t_{1}-s_{1})^{\alpha+\beta-1}(s_{2}-s_{1})^{1-(\alpha+\beta)}.
$$

In this case
$$
|\psi(s_{2})-\psi(s_{1})|\leq
\frac{M\Gamma(\beta)}{\Gamma(\alpha+\beta)}(t_{1}-s_{2})^{\alpha+\beta-1}(t_{1}-s_{1})^{\alpha+\beta-1}
(s_{2}-s_{1})^{1-(\alpha+\beta)}
$$
\begin{equation}\label{equ6}
+\frac{2M}{\Gamma(\alpha+1)}(t_{1}-s_{2})^{\beta-1}(s_{2}-s_{1})^{\alpha}.
\end{equation}%

From inequalities (\ref{equ5}) and (\ref{equ6}) it follows that the function $\psi$
is continuous on the segment $[t_{0}, \ t_{1}].$
Therefore, the function $Sa$, as the product of two continuous
functions, is continuous. \qquad\qquad\qquad\qquad\qquad\qquad\qquad\qquad\qquad\qquad\quad$\square$
\end{proof}
\begin{lemma}\label{lem33}
Suppose that $0< \alpha \leq 1,\beta>0$ are
real numbers and $a_{0},\ a_{1} $  $ \in C([t_{0}, t_{1}],
\mathbb{R}^{n}).$ Then to fulfill equality
\begin{equation}\label{equ7}
\int\limits_{t_{0}}^{t_{1}}(t_{1}-t)^{\beta-1}\left[\langle
a_{0}(t),h(t)\rangle+\langle
a_{1}(t),(^{c}D_{t_{0}+}^{\alpha}h)(t)\rangle \right]dt=0, \,\,
\forall h \in C_{0}^{\alpha}([t_{0}, t_{1}], \mathbb{R}^{n}),
\end{equation}%
it is necessary and sufficient that

1)
$(t_{1}-t)^{1-\alpha}(I_{t_{1}-}^{\alpha}b)(t)+(t_{1}-t)^{\beta-\alpha}a_{1}(t)=0$
on $t \in [t_{0}, t_{1}],$ if $\beta>\alpha>0$,

2)
$(t_{1}-t)^{1-\beta}(I_{t_{1}-}^{\alpha}b)(t)+a_{1}(t)=\frac{k}{\Gamma(\alpha)}(t_{1}-t)^{\alpha-\beta}$
on $t \in [t_{0}, t_{1}],$ if $0<\beta \leq \alpha \leq 1$, for
some constant $0\neq k \in \mathbb{R}^{n},$ where
$b(t)=(t_{1}-t)^{\beta-1}a_{0}(t).$
\end{lemma}
\begin{proof} For $h \in C_{0}^{\alpha}([t_{0}, t_{1}],
\mathbb{R}^{n})$ we put $h(t)=\varphi(t)r,$ where $\varphi \in
C_{0}^{\alpha}([t_{0}, t_{1}], \mathbb{R})$ and $r \in
\mathbb{R}^{n}.$ Given the representation
$$
\varphi(t)=\frac{1}{\Gamma(\alpha)}
\int\limits_{t_{0}}^{t}(t-\tau)^{\alpha-1}(^{c}D^{\alpha}_{t_{0}+}\varphi)(\tau)d\tau,
$$
we transform (\ref{equ7}) as follows:
$$
\int\limits_{t_{0}}^{t_{1}}(t_{1}-t)^{\beta-1}\left[\langle
a_{0}(t), r \rangle \frac{1}{\Gamma(\alpha)}
\int\limits_{t_{0}}^{t}(t-\tau)^{\alpha-1}(^{c}D_{t_{0}+}^{\alpha}\varphi)(\tau)d\tau
+\langle a_{1}(t), r
\rangle(^{c}D_{t_{0}+}^{\alpha}\varphi)(t)\right]dt
$$
\begin{eqnarray}\label{equ8}
&=\int\limits_{t_{0}}^{t_{1}}(t_{1}-t)^{\beta-1}
\left\langle\frac{(t_{1}-t)^{1-\beta}}{\Gamma(\alpha)}\int\limits_{t}^{t_{1}}(t_{1}-\tau)^{\beta-1}
(\tau-t)^{\alpha-1}a_{0}(\tau)d\tau +a_{1}(t), r
\right\rangle\nonumber\\
&\times(^{c}D_{t_{0}}^{\alpha}\varphi)(t)dt=0.
\end{eqnarray}%

By Lemma \ref{lem32}, the function
$f(t)=(t_{1}-t)^{1-\beta}(I_{t_{1}-}^{\alpha}b)(t)+a_{1}(t)$ is
continuous on the interval $[t_{0}, t_{1}].$ Then by virtue of
Lemma \ref{lem31} we have

1) $\langle
(t_{1}-t)^{1-\alpha}(I_{t_{1}-}^{\alpha}b)(t)+(t_{1}-t)^{\beta-\alpha}a_{1}(t),
r \rangle=0,$ if $\beta > \alpha>0,$

2) $\langle (t_{1}-t)^{1-\beta}(I_{t_{1}-}^{\alpha}b)(t)+a_{1}(t),
r \rangle=\frac{k_{0}}{\Gamma(\alpha)}(t_{1}-t)^{\alpha-\beta},$
if $0<\beta \leq \alpha \leq 1,$ where $k_{0}$ is a constant.

Since $r \in \mathbb{R}^{n}$ is arbitrary, the assertions of Lemma
\ref{lem33} follow from Lemma \ref{lem31}. \qquad\qquad\qquad\qquad\qquad\qquad\qquad\qquad\qquad\qquad\qquad\qquad\qquad\qquad\qquad\quad$\square$
\end{proof}

\section{Necessary conditions of Euler-Lagrange and Legendre for
fractional variational problem}\label{sec4}

\subsection {The simplest problem of fractional calculus of variations}

 For $0<\alpha \leq 1$ and $\beta>0$ we call the simplest problem
 of fractional calculus of variations the following extremal
 problem in the space $C^{\alpha}([t_{0}, t_{1}], \mathbb{R}^{n})$:
$$
(P)\;\;J(x(\cdot))=\int\limits_{t_{0}}^{t_{1}}(t_{1}-t)^{\beta-1}L(t,
x(t), (^{c}D_{t_{0}+}^{\alpha}x)(t))dt \rightarrow extr, \\
x(t_{0})=x_{0}, \,x(t_{1})=x_{1}.
$$

Here the segment $[t_{0}, t_{1}]$ is assumed to be fixed and
finite, $t_{0}<t_{1}. L$ $=L(t, x, y)$ is a function $2n+1$
variables, called the integrand. The extremum in problem (P) is
considered among the space functions $C^{\alpha}([t_{0},
t_{1}],\mathbb{R}^{n})$ satisfying the conditions at the ends or
boundary conditions: $x(t_{0})=x_{0}$, $x(t_{1})=x_{1}.$ Such
functions are called admissible.
\begin{definition}\label{def41}
We will say that an admissible function
$x^{0}$ deliver a weak local minimum in problem (P), and write
$x^{0} \in wlocmin P$ if there exists $\delta>0$ such that
$J(x(\cdot))\geq J(x^{0}(\cdot))$ for any admissible function $x$
for which
$$
\|x(\cdot)-x^{0}(\cdot)\|_{C^{\alpha}([t_{0},t_{1}],
\mathbb{R}^{n})}< \delta,
$$
where
$$
\|y\|_{C^{\alpha}([t_{0},t_{1}], \mathbb{R}^{n})}=\max\limits_{t
\in [t_{0}, t_{1}]}\|y(t)\|+\max\limits_{t \in [t_{0},
t_{1}]}\|(^{c}D_{t_{0}+}^{\alpha}y)(t)\|.
$$
\end{definition}
\begin{theorem}\label{teo41}
Let $0< \alpha \leq 1,$ $\beta>0$ and the
function $x^{0}$ deliver a weak local minimum in the problem (P)
$(x^{0} \in wlocmin P),$ the functions $L, \ L_{x}, \ L_{y}$ are
continuous in some neighborhood of the graph $\{(t, x^{0}(t),
(^{c}D_{t_{0}+}^{\alpha}x^{0})(t))|t \in [t_{0}, t_{1}]\}.$ Then\\
1) if $0< \alpha < \beta$, then the function $x^{0}$ is a solution
to the equation
\begin{equation}\label{equ9}
(t_{1}-t)^{1-\alpha}(I_{t_{1}-}^{\alpha}b)(t)+(t_{1}-t)^{\beta-\alpha}L_{y}(t,
x^{0}(t),(^{c}D_{t_{0}+}^{\alpha}x^{0})(t))=0,\; t\in[t_{0}, t_{1}]
\end{equation}%
2) if $0<\beta \leq \alpha \leq 1,$ then function $x^{0}$ is a
solution to the equation
\begin{equation}\label{equ10}
(t_{1}-t)^{1-\beta}(I_{t_{1}-}^{\alpha}b)(t)+L_{y}(t,
x^{0}(t),(^{c}D_{t_{0}+}^{\alpha}x^{0})(t))=\frac{k}{\Gamma(\alpha)}(t_{1}-t)^{\alpha-\beta},\; t\in[t_{0}, t_{1}]
\end{equation}%
for some
constant $0\neq k \in \mathbb{R}^{n},$  where
$b(t)=(t_{1}-t)^{\beta-1}L_{x}(t, x^{0}(t),
(^{c}D_{t_{0}+}^{\alpha}x^{0})(t))$.
\end{theorem}

\begin{proof} Let's take an arbitrary but fixed function $h \in
C_{0}^{\alpha}([t_{0}, t_{1}], \mathbb{R}^{n})$. Since $x^{0} \in
wlocmin P$ then the function of one variable
\begin{eqnarray*}
&\phi(\lambda)=J(x^{0}(\cdot)+\lambda
h(\cdot))\\
&=\int\limits_{t_{0}}^{t_{1}}(t_{1}-t)^{\beta-1}L(t,
x^{0}(t)+\lambda h(t),
(^{c}D_{t_{0}+}^{\alpha}x^{0})(t)+\lambda(^{c}D_{t_{0}+}^{\alpha}h)(t))dt
\end{eqnarray*}
has an extremum at $\lambda=0$. From the smoothness conditions
imposed on $L, \, x^{0}, \ h,$ it follows that the function $\phi$
is differentiable at zero. But then, according to Fermat's
theorem, $\phi'(0)=0.$ Differentiating the function $\phi$ and
assuming $\lambda=0$ we obtain

$$
\phi'(0)=\int\limits_{t_{0}}^{t_{1}}(t_{1}-t)^{\beta-1} [\langle
L_{x}(t, x^{0}(t), (^{c}D_{t_{0}+}^{\alpha}x^{0})(t)), h(t)
\rangle
$$
$$
+\langle L_{y}(t, x^{0}(t), (^{c}D_{t_{0}+}^{\alpha}x^{0})(t)),
(^{c}D_{t_{0}+}^{\alpha}h)(t)\rangle]dt=0,\,\,\, h \in
C_{0}^{\alpha}([t_{0}, t_{1}], \mathbb{R}^{n}).
$$

If we denote $a_{0}(t)=L_{x}(t, x^{0}(t),
(^{c}D_{t_{0}+}^{\alpha}x^{0})(t))$ and $a_{1}(t)=L_{y}(t,
x^{0}(t),$ $(^{c}D_{t_{0}+}^{\alpha}x^{0})(t))$ here, then
equalities (\ref{equ9}) and (\ref{equ10}) follow from Lemma \ref{lem33}. \qquad$\square$
\end{proof}
\begin{example}\label{exam41}
For $0<\alpha \leq 1$ and $\beta>0$ we
consider the following extremal problem in the space
$C^{\alpha}([0, 1], \mathbb{R}):$
\begin{equation}\label{equ11}
J(x(\cdot))=\int\limits_{0}^{1}(1-t)^{\beta-1}((^{c}D_{0+}^{\alpha}x)(t))^{2}dt \rightarrow extr,\; x(0)=0, \,x(1)=1.
\end{equation}%

The methods of classical calculus of variations are not applicable
here since the integrand $(1-t)^{\beta-1}y^{2},$ for $\beta<1$ is
not continuous at $t=1$. To solve the problem we use Theorem \ref{teo41}.

Let $0<\alpha \leq 1$ and $\beta> \alpha > 0$, then by condition
(\ref{equ9}) we have
$$
2(^{c}D_{0+}^{\alpha}x)(t)=0, \,\,\,\,\, t \in [0, 1].
$$

Hence we get $x(t)=x(0)=0, \,\,\,\, t \in [0, 1].$ Therefore,
problem (\ref{equ11}) does not have solution.

Now consider the case $0<\beta \leq \alpha \leq 1$. In this case
equation (\ref{equ10}) takes the form
$$
2(^{c}D_{0+}^{\alpha}x)(t)=\frac{k}{\Gamma(\alpha)}(1-t)^{\alpha-\beta},
\,\,\,\,\, t \in [0, 1].
$$

The solution to this equation is
$$
x(t)=x(0)+\frac{k}{2\Gamma^{2}(\alpha)}\int\limits_{0}^{t}(t-\tau)^{\alpha-1}(1-\tau)^{\alpha-\beta}d\tau,
\,\,\, t \in [0, 1].
$$

Now with the boundary conditions we determine
$$
k=2(2\alpha-\beta)\Gamma^{2}(\alpha).
$$

Thus, the extremal has the form
$$
x(t)=(2\alpha-\beta)\int\limits_{0}^{t}(t-\tau)^{\alpha-1}(1-\tau)^{\alpha-\beta}d\tau,
\,\, t \in [0, \ 1].
$$

Hence, for $\beta=\alpha <1,$ as in work \cite{14}, we obtain
$x(t)=t^{\alpha}.$ And if $\beta=\alpha=1,$ then, as in the
classical calculus of variations, we obtain $x(t)=t,$ $t \in [0, \
1].$

Now, for the function$L(y)=y^{2},$ using the inequality
$$
L(z)-L(x)\geq 2x(z-x), \,\,\, \forall x, z \in R,
$$
we show that function
$x(t)=(2\alpha-\beta)\int\limits_{0}^{t}(t-\tau)^{\alpha-1}(1-\tau)^{\alpha-\beta}d\tau$
is indeed a solution to problem (\ref{equ11}). Then for $z \in
C^{\alpha}([0, 1], \mathbb{R})$ and
$x(t)=(2\alpha-\beta)\int\limits_{0}^{t}(t-\tau)^{\alpha-1}(1-\tau)^{\alpha-\beta}d\tau,$
we get
$$
J(z(\cdot))-J(x(\cdot))=\int\limits_{0}^{1}(1-t)^{\beta-1}[((^{c}D_{0+}^{\alpha}z)(t))^{2}-((^{c}D_{0+}^{\alpha}x)(t))^{2}]dt
$$
$$
\geq
2\int\limits_{0}^{1}(1-t)^{\beta-1}(^{c}D_{0+}^{\alpha}x)(t)[(^{c}D_{0+}^{\alpha}z)(t)-(^{c}D_{0+}^{\alpha}x)(t)]dt
$$
$$
=\int\limits_{0}^{1}(1-t)^{\beta-1}\frac{k}{\Gamma(\alpha)}(1-t)^{\alpha-\beta}[(^{c}D_{0+}^{\alpha}z)(t)-(^{c}D_{0+}^{\alpha}x)(t)]dt
$$
$$
=\frac{k}{\Gamma(\alpha)}\int\limits_{0}^{1}(1-t)^{\alpha-1}(^{c}D_{0+}^{\alpha}(z-x))(t)dt=k[(z(1)-x(1))-(z(0)-x(0))]=0.
$$

Thus, the function
$x(t)=(2\alpha-\beta)\int\limits_{0}^{t}(t-\tau)^{\alpha-1}(1-\tau)^{\alpha-\beta}d\tau$
is indeed a solution to problem (\ref{equ11}).

We will consider problem (P) and derive Legendre's conditions
using the second variation of the functional. We assume that $L$
is a continuous function with continuous partial derivatives
$L_{x}, \, L_{y}, \, L_{xx}, \,L_{xy}, \,L_{yy}$. The solution to
problem (P) is sought in the function space $PC^{\alpha}([t_{0},
t_{1}], \mathbb{R}^{n}).$ Let the function $x^{0}\in
PC^{\alpha}([t_{0}, t_{1}], \mathbb{R}^{n})$ deliver a weak local
minimum to the functional $J(\cdot).$ Then standard calculations
show that for any (fixed) variation $h \in PC_{0}^{\alpha}([t_{0},
t_{1}], \mathbb{R}^{n}),$ the condition
\begin{eqnarray}\label{equ12}
&\delta^{2}J(x^{0}(\cdot),
h(\cdot))=\frac{d^{2}}{d\lambda^{2}}J(x^{0}+\lambda
h)|_{\lambda=0}= \int\limits_{t_{0}}^{t_{1}}(t_{1}-t)^{\beta-1}
[\langle P(t)(^{c}D_{t_{0}+}^{\alpha}h)(t),
\nonumber\\
&(^{c}D_{t_{0}+}^{\alpha}h)(t)\rangle +2\langle Q(t)h(t), (^{c}D_{t_{0}+}^{\alpha}h)(t)\rangle+\langle
R(t)h(t), h(t)\rangle]dt\geq 0,
\end{eqnarray}%
where \, $R(t)=L_{xx}(t, x^{0}(t),
(^{c}D_{t_{0}+}^{\alpha}x^{0})(t)),$\, $Q(t)=L_{xy}(t,
x^{0}(t)(^{c}D_{t_{0}+}^{\alpha}x^{0})(t)),$\, $P(t)=L_{yy}(t,
x^{0}(t),$ $(^{c}D_{t_{0}+}^{\alpha}x^{0})(t)).$

Let $T_{1}\subset (t_{0}, t_{1})$ is the set of continuity point
of the function $^{c}D_{t_{0}+}^{\alpha}x^{0}.$
\end{example}
\begin{theorem}\label{teo42}
Let $0< \alpha \leq1,$ $\beta>0$ and the
function $x^{0}\in PC^{\alpha}([t_{0}, t_{1}], \mathbb{R}^{n})$
deliver a weak local minimum in the problem (P) $(x^{0} \in
\texttt{wloc}\min P)$, the functions $L, \,L_{x},
\,L_{y},\,L_{xx}, \,L_{xy}$ and $L_{yy}$ are continuous in some
neighborhood of the graph $\{(t, x^{0}(t),
(^{c}D_{t_{0}+}^{\alpha}x^{0})(t)|[t_{0}, t_{1}]\}.$ Then the
following inequality holds:
\begin{equation}\label{equ13}
\langle P(t)r, \,r\rangle \geq 0, \,\,\,  \forall t \in T_{1},
\,\, \forall r \in R^{n}. 
\end{equation}%
\end{theorem}
\begin{proof}
 Let the function $x^{0}$ give the problem (P) the
minimum value. To prove the theorem, we will argue by
contradiction. Let us assume that there exists a point $\sigma \in
T_{1}$ and a vector $r \in R^{n}$ for which the inequality
\begin{equation}\label{equ14}
\langle P(\sigma)r, \,r\rangle < 0
\end{equation}%
holds. Then by virtue of (\ref{equ14}) and the continuity of the function
$P(t)$, in a neighborhood of the point $\sigma$ there will be a
sufficiently small number $\varepsilon>0$ and number $\gamma>0$
such that
\begin{equation}\label{equ15}
\langle P(t)r, \,r\rangle < -\gamma<0,\,\,\, t \in
[\sigma-\varepsilon, \sigma+\varepsilon]\subset (t_{0}, t_{1}).
\end{equation}%

Let $f(t),$ $t \in [\sigma-\varepsilon, \sigma+\varepsilon]$ be a
continuous non-constant function, $k$ is an unknown real number
and $M=\max \limits_{t \in [\sigma-\varepsilon,
\sigma+\varepsilon]}|f(t)|.$ Define the function
\begin{equation}\label{equ16}
h(t)=\frac{1}{\Gamma(\alpha)}\int\limits_{t_{0}}^{t}(t-\tau)^{\alpha-1}g(\tau)d\tau,
\,\,\, t \in [t_{0}, t_{1}],
\end{equation}%
where
$$
g(t)=\left\{
\begin{array}{c}
(f(t)-k)r,\, \, \, \, \, \, \,\, \,\,\,\,\,  t \in
[\sigma-\varepsilon,
\sigma+\varepsilon], \\
\\
0,  \, \, \, \, \,\,\,\,\,\,\,\,\,\,\,\,\, \, t \in [t_{0},
t_{1}]\backslash[\sigma-\varepsilon, \sigma+\varepsilon].
\end{array}\right.
$$

Then it is obvious that $h(t_{0})=0$ and
$(^{c}D_{t_{0}+}^{\alpha}h)(t)=g(t),$ $t \in [t_{0}, t_{1}].$
Let's determine the real number $k$ from the condition
$h(t_{1})=0.$ Then we have
$$
k=\left(\alpha
\int\limits_{\sigma-\varepsilon}^{\sigma+\varepsilon}(t_{1}-t)^{\alpha-1}f(t)dt\right)/
\left((t_{1}-(\sigma-\varepsilon))^{\alpha}-(t_{1}-(\sigma+\varepsilon))^{\alpha}\right).
$$

Therefore $h \in PC_{0}^{\alpha}([t_{0}, t_{1}], \mathbb{R}^{n}).$
By the mean value theorem $k=f(c)$, where $c \in
(\sigma-\varepsilon, \sigma+\varepsilon).$ From (\ref{equ16}), for $t \in
[\sigma-\varepsilon, \sigma+\varepsilon]$ we have
\begin{equation}\label{equ17}
\|h(t)\|\leq
\frac{1}{\Gamma(\alpha)}\int\limits_{\sigma-\varepsilon}^{t}(t-\tau)^{\alpha-1}|f(\tau)-f(c)|
\|r\|d\tau \leq \frac{2M
\|r\|}{\Gamma(\alpha+1)}(t-(\sigma-\varepsilon))^{\alpha},
\end{equation}%
and for $t \in [\sigma+\varepsilon, t_{1}]$ we have
\begin{eqnarray}\label{equ18}
&\|h(t)\|\leq
\frac{1}{\Gamma(\alpha)}\int\limits_{\sigma-\varepsilon}^{\sigma+\varepsilon}(t-\tau)^{\alpha-1}|f(\tau)-f(c)|
\|r\|d\tau  \nonumber\\
&\leq \frac{2M
\|r\|}{\Gamma(\alpha+1)}[(t-(\sigma-\varepsilon))^{\alpha}-(t-(\sigma+\varepsilon))^{\alpha}].
\end{eqnarray}%

Now, taking into account (\ref{equ16})-(\ref{equ18}), we will estimate each of the three terms included in inequality (\ref{equ12}). For this, for convenience, in each estimate the left-hand side of the resulting inequality will be denoted by $\Omega_k, k=1,2,3$. In other words, $\Omega_k$ is the absolute value of the k-th term in (\ref{equ12}). 

Thus, using inequality (\ref{equ15}), the first term of (\ref{equ12}) can be estimated as follows: 
$$
\int\limits_{t_{0}}^{t_{1}}(t_{1}-t)^{\beta-1}\langle P(t)(^{c}D_{t_{0}+}^{\alpha}h)(t),\,
(^{c}D_{t_{0}+}^{\alpha}h)(t)\rangle dt
$$
$$
=\int\limits_{\sigma-\varepsilon}^{\sigma+\varepsilon}(t_{1}-t)^{\beta-1}
(f(t)-f(c))^{2}\langle P((t)r, \, r\rangle dt
$$

\[\leq\left\{
\begin{array}{c}
-8M^{2}\gamma(t_{1}-t_{0})^{\beta-1}\varepsilon, \, \, \, \, \, \,
\,\, \,\,\,   0<\beta \leq 1, \\
\\
-8M^{2}\gamma(t_{1}-\xi)^{\beta-1}\varepsilon, \, \, \, \,
\,\,\,\,\,\,\,\,\,\,\,\, \beta>1,
\end{array} \right.\]

where $\sigma+\varepsilon<\xi<t_{1}$.  
Finally, we have
 \begin{equation}\label{equ19}
\Omega_1\leq -M_{1}\gamma \varepsilon; \; M_{1}=\min\{8M^{2}(t_{1}-t_{0})^{\beta-1},
8M^{2}(t_{1}-\xi)^{\beta-1}\}.
\end{equation}%
The second term of (\ref{equ12}) can be bounded as follows
$$
\left|2\int\limits_{t_{0}}^{t_{1}}(t_{1}-t)^{\beta-1}\langle Q(t)h(t),\,
(^{c}D_{t_{0}+}^{\alpha}h)(t)\rangle dt \right|
$$
$$
\leq \frac{8M^{2}\|Q\|\|r\|^{2}}{\Gamma(\alpha+1)}
\int\limits_{\sigma-\varepsilon}^{\sigma+\varepsilon}(t_{1}-t)^{\beta-1}
(t-(\sigma-\varepsilon))^{\alpha}dt
$$

\[\leq\left\{
\begin{array}{c}
K_{1}(t_{1}-\xi)\varepsilon^{\alpha+1}, \, \, \, \, \, \, \,\, \,\,
 0<\beta \leq
1,  \\
\\
K_{1}(t_{1}-t_{0})^{\beta-1}\varepsilon^{\alpha+1}, \, \, \, \, \,
\,\,\,\,\,\,  \beta>1,
\end{array} \right.\]
where $K_{1}=\frac{2^{\alpha+4}M^{2}\|Q\|
\|r\|^{2}}{\Gamma(\alpha+2)}, \|Q\|=\max\limits_{t \in
[t_{0}, t_{1}]}\|Q(t)\|$. Then we have
\begin{equation}\label{equ20}
\Omega_2\leq M_{2}\varepsilon^{\alpha+1};\; M_{2}=\max
\big\{K_{1}(t_{1}-\xi)^{\beta-1},\; K_{1}(t_{1}-t_{0})^{\beta-1}\big\}.
\end{equation}%
We will write the third term of (\ref{equ12}) as the sum of two terms. Taking into
account the estimates of (\ref{equ17}), we evaluate the first of them as
follows
$$
\left|\int\limits_{\sigma-\varepsilon}^{\sigma+\varepsilon}(t_{1}-t)^{\beta-1}
\langle R(t)h(t), \, h(t)\rangle dt \right| \leq \frac{4M^{2}\|R\|
\|r\|^{2}}{\Gamma^{2}(\alpha+1)}
\int\limits_{\sigma-\varepsilon}^{\sigma+\varepsilon}(t_{1}-t)^{\beta-1}
(t-(\sigma-\varepsilon))^{2\alpha}dt
$$

\[\leq\left\{
\begin{array}{c}
K_{2}(t_{1}-\xi)^{\beta-1}\varepsilon^{2\alpha+1},
 \,\, \,\, \, \,\,   0<\beta \leq 1,\, \, \\
\\
K_{2}(t_{1}-t_{0})^{\beta-1}\varepsilon^{2\alpha+1}, \, \, \, \, \,
\,\,\, \, \, \,\,\,\, \,  \beta>1,
\end{array}\right.\]
where $K_{2}=\frac{2^{2\alpha+3}M^{2}\|R\|
\|r\|^{2}}{(2\alpha+1)\Gamma^{2}(\alpha+1)}$, $\|R\|=\max\limits_{t \in
[t_{0}, t_{1}]}\|R(t)\|$. \\
 As a result we derive
 \begin{equation}\label{equ21}
 \Omega_3\leq M_{3}\varepsilon^{2\alpha+1};\;\; M_{3}=\max
\left\{K_{2}(t_{1}-\xi)^{\beta-1},
\,K_{2}(t_{1}-t_{0})^{\beta-1}\right\}.
\end{equation}%

Using Lemma \ref{lem21} and taking into account the estimates of (\ref{equ18}), the
second term can be bounded as follows
$$
\left|\int\limits_{\sigma+\varepsilon}^{t_{1}}(t_{1}-t)^{\beta-1}\langle R(t)h(t),
\, h(t)\rangle dt \right|
$$
$$
\leq \frac{4M^{2}\|R\| \|r\|^{2}}{\Gamma^{2}(\alpha+1)}
\int\limits_{\sigma+\varepsilon}^{t_{1}}(t_{1}-t)^{\beta-1}
[(t-(\sigma-\varepsilon))^{\alpha}-(t-(\sigma+\varepsilon))^{\alpha}]^{2}dt
$$
$$
\leq \frac{2^{\alpha+3}\alpha M^{2}\|R\|
\|r\|^{2}}{\Gamma^{2}(\alpha+1)}(t_{1}-(\sigma+\varepsilon))^{\alpha+\beta-1}
\int\limits_{0}^{1}(1-z)^{\beta-1}z^{\alpha-1}dz\cdot\varepsilon^{1+\alpha}
$$
$$
=K_{3}
(t_{1}-(\sigma+\varepsilon))^{\alpha+\beta-1}\varepsilon^{1+\alpha}
$$

\[\leq\left\{
\begin{array}{c}
K_{3}(t_{1}-\xi)^{\alpha+\beta-1}\varepsilon^{1+\alpha},
\, \, \, \, \, \, \,\, \,\,   0<\alpha+\beta \leq 1,\, \,\\
\\
K_{3}(t_{1}-t_{0})^{\alpha+\beta-1}\varepsilon^{1+\alpha}, \, \, \,
\, \, \,\,\,  \alpha+\beta>1,\, \, \, \, \, \, \,\, \, \,
\, \, \, \,\, \,
\end{array} \right.\]
where $K_{3}=\frac{2^{\alpha+3}M^{2}\|R\|
\|r\|^{2}\Gamma(\beta)}{\Gamma(\alpha+1)\Gamma(\alpha+\beta)}$. Hence, we get
\begin{equation}\label{equ22}
\Omega_4\leq
M_{4}\varepsilon^{1+\alpha};\;\; M_{4}=\max \left\{K_{3}(t_{1}-\xi)^{\alpha+\beta-1},
\,K_{3}(t_{1}-t_{0})^{\alpha+\beta-1}\right\}.
\end{equation}%
Taking into account estimates (\ref{equ19})-(\ref{equ22}) in (\ref{equ12}),
we arrive at the inequality
$$
\delta^{2}J(x^{0}(\cdot),
h(\cdot))\leq\varepsilon[-M_{1}\gamma+\varepsilon^{\alpha}(M_{2}+M_{3}\varepsilon^{\alpha}+M_{4})],
$$
which for sufficiently small $\varepsilon>0$ takes the form
$\delta^{2}J(x^{0}(\cdot), h(\cdot))<0,$ which contradicts the
definition of optimality $\delta^{2}J(x^{0}(\cdot), h(\cdot))\geq
0.$  \qquad\qquad\qquad$\square$
\end{proof}

\subsection {Problem with free initial condition and fixed final
condition}
Let $0< \alpha \leq 1$ and $\beta>0$.  Let's consider the
following problem in space $C^{\alpha}([t_{0}, t_{1}], \,
\mathbb{R}^{n}):$
\begin{eqnarray*}
(P_{1})\quad &&J(x(\cdot))=\int \limits_{t_{0}}^{t_{1}}(t_{1}-t)^{\beta-1}L(t,
x(t),(^{c}D_{t_{0}+}^{\alpha}x)(t))dt+ l(x(t_{0}))\rightarrow
extr, \\
&&x(t_{1})=x_{1}.
\end{eqnarray*}%

Here the segment $[t_{0}, t_{1}]$ is assumed to be fixed and
finite, $t_{0}< t_{1}.$ $L(t, x, y)$ is a continuous function with
continuous partial derivatives $L_{x}$ and $L_{y}$, $l$ is a
continuous function with continuous derivative $l_{x_0}$.
\begin{theorem}\label{theo43}
Let $0< \alpha \leq 1,$ $\beta>0$ and the
function $x^{0}$ deliver a weak local minimum in the problem
(P$_{1}$)($x^{0}\in \texttt{wloc}\min P_{1}$), the functions $L,
L_{x}, L_{y}$ are continuous in some neighborhood of the graph
$\{(t, x^{0}(t), (^{c}D_{t_{0}+}^{\alpha}x^{0})(t))| \, t \in
[t_{0}, t_{1}]\}$, the function $l$ is continuously differentiable
in the neighborhood of the point $x^{0}(t_{0}).$ Then

1) if $0< \alpha < \beta$, then the function $x^{0}$ is a solution
to the problem
\begin{eqnarray}\label{equ23}
(t_{1}-t)^{1-\alpha}&&(I_{t_{1}-}^{\alpha}b)(t)+(t_{1}-t)^{\beta-\alpha}L_{y}(t)=0,
\ \ \ \ t\in \left[t_0,\ t_1\right], \nonumber\\
&&\int\limits_{t_{0}}^{t_{1}}(t_{1}-t)^{\beta-1}L_{x}(t)dt+l_{x_{0}}=0,\\
&&x(t_{1})=x_{1}\nonumber, 
\end{eqnarray}%

2) if $0<\beta \leq \alpha \leq 1,$ then function $x^{0}$ is a
solution to the problem
\begin{eqnarray}\label{equ24}
(t_{1}-t)^{1-\beta}&&(I_{t_{1}-}^{\alpha}b)(t)+L_{y}(t)+k\frac{(t_{1}-t)^{\alpha-\beta}}{\Gamma(\alpha)}=0,
\ \ \ \ t\in \left[t_0,\ t_1\right],  \nonumber\\
&&\int\limits^{t_1}_{t_0}{{(t_1-t)}^{\beta
-1}L_x(t)dt}+l_{x_{0}}+k=0, \\
&&x(t_{1})=x_{1},\nonumber
\end{eqnarray}%
where $0 \neq k \in \mathbb{R}^{n}$ some constant, $b(t)=(t_{1}-t)^{\beta-1}L_{x}(t, x^{0}(t),
(^{c}D_{t_{0}+}^{\alpha}x^{0})(t))$.
\end{theorem}
\begin{proof} Let $0< \alpha \leq 1$ and $\beta>0$. Under our
assumptions on $L, L_{x}, L_{y}, l, l_{x_{0}}$, we can conclude
that for the solution $x^{0} \in C^{\alpha}([t_{0}, t_{1}], \,
\mathbb{R}^{n})$ of the problem (P$_{1}$) and any (fixed)
variation $h \in C^{\alpha}([t_{0}, t_{1}], \, \mathbb{R}^{n})$,
$h(t_{1})=0$ the first variation $\delta J (x^{0}(\cdot),
h(\cdot))=\frac{d}{d \lambda}J(x^{0}+\lambda h)|_{\lambda=0}=0.$
Standard calculations show that
\begin{eqnarray}\label{equ25}
&\int\limits_{t_{0}}^{t_{1}}(t_{1}-t)^{\beta-1}\left[\langle
L_{x}(t), h(t)\rangle+\langle L_{y}(t),
(^{c}D_{t_{0}+}^{\alpha}h)(t)\rangle \right]dt \nonumber\\
&+\langle l_{x_{0}}, h(t_{0})\rangle=0, \,\,\, \forall h \in
C^{\alpha}([t_{0}, t_{1}], \, \mathbb{R}^{n}), h(t_{1})=0.
\end{eqnarray}%
For $h(t)=\varphi(t)r,$ where $\varphi \in C^{\alpha}([t_{0},
t_{1}], \ \mathbb{R}),$ $\varphi(t_{1})=0$ and $r \in
\mathbb{R}^{n},$ we write the equality (\ref{equ25}) in the form
\begin{eqnarray}\label{equ26}
&\int\limits_{t_{0}}^{t_{1}}(t_{1}-t)^{\beta-1}\langle
(t_{1}-t)^{1-\beta}(I_{t_{1}-}^{\alpha}b)(t) +L_{y}(t), r\rangle (^{c}D_{t_{0}+}^{\alpha}\varphi)(t)dt\nonumber\\
&+\left\langle
\int\limits_{t_{0}}^{t_{1}}(t_{1}-t)^{\beta-1}L_{x}(t)dt+l_{x_{0}},
r \right \rangle \varphi(t_{0})=0.
\end{eqnarray}%

If $ \varphi(t_{0})=0$ and $0< \beta \leq \alpha \leq 1$, then by
virtue of Lemma \ref{lem31} we have
\begin{equation}\label{equ27}
\langle ((t_{1}-t)^{1-\beta}(I_{t_{1}-}^{\alpha}b)(t) +L_{y}(t)),
r\rangle =\frac{k_{0}}{\Gamma(\alpha)} (t_{1}-t)^{\alpha-\beta},
\end{equation}%
where $k_{0}$ is some real number.

Taking this equality into account in (\ref{equ26}) we get
$$
\frac{k_{0}}{\Gamma(\alpha)}\int\limits_{t_{0}}^{t_{1}}(t_{1}-t)^{\alpha-1}
(^{c}D_{t_{0}+}^{\alpha}\varphi)(t)dt+\left\langle
\int\limits_{t_{0}}^{t_{1}}(t_{1}-t)^{\beta-1}L_{x}(t)dt+l_{x_{0}},
r \right\rangle \varphi(t_{0})
$$
\begin{equation}\label{equ28}
=k_{0}\varphi(t_{1})+\left[\left\langle
\int\limits_{t_{0}}^{t_{1}}(t_{1}-t)^{\beta-1}L_{x}(t)dt+l_{x_{0}},
r \right\rangle -k_{0}\right]\varphi(t_{0})=0.
\end{equation}%

Then from (\ref{equ27}) and (\ref{equ28}) it follows that for $0< \beta \leq \alpha
\leq 1$ the solution $x^{0}$ of problem (P) must satisfy system
(\ref{equ24}).

Under the condition $\varphi(t_{0})=0$ and $0<\alpha < \beta$ from
(\ref{equ26}) using Lemma \ref{lem31} we obtain
\begin{equation}\label{equ29}
[(t_{1}-t)^{1-\alpha}(I_{t_{1}-}^{\alpha}b)(t)+(t_{1}-t)^{\beta-\alpha}L_{y}(t)]=0.
\end{equation}%

Then from (\ref{equ26}) and (\ref{equ29}) it follows that for $0< \alpha <\beta$ the
function $x^{0}$ must be the solution to problem (\ref{equ23}). \qquad\qquad\qquad\qquad\qquad\qquad\qquad\qquad\qquad\quad$\square$
\end{proof}
\begin{example}\label{exam42}
For $0< \alpha \leq 1$ and $\beta>0$ we
consider the following extremal problem in the space
$C^{\alpha}([0, 1], \ \mathbb{R}):$
\begin{equation}\label{equ30}
J(x(\cdot))=\int\limits_{0}^{1}(1-t)^{\beta-1}(^{c}D_{0+}^{\alpha}x)^{2}(t)dt+x^{2}(0)\rightarrow
extr, \ \ \ \ \ x(1)=1.
\end{equation}%

Let $0< \beta \leq \alpha \leq 1.$ In this case by condition (\ref{equ24})
we have
\begin{eqnarray}
&(^{c}D_{0+}^{\alpha}x)(t)=-k
\frac{(1-t)^{\alpha-\beta}}{2\Gamma(\alpha)}, \ \ \ t \in [0,\ 1], \ \ \label{equ31}\nonumber\\
&2x(0)+k=0, \ \ \ x(1)=1. 
\end{eqnarray}

Solution of equation in (\ref{equ31})
$$
x(t)=x(0)-\frac{k}{2\Gamma^{2}(\alpha)}\int\limits_{0}^{t}(t-\tau)^{\alpha-1}(1-\tau)^{\alpha-\beta}d\tau,
\,\,\, t \in [0, \ 1].
$$

Taking into account this solution in the system (\ref{equ31}) we find
$$
k=-\frac{2(2\alpha-\beta)\Gamma^{2}(\alpha)}{1+(2\alpha-\beta)\Gamma^{2}(\alpha)},
\,\,\,
x(0)=\frac{(2\alpha-\beta)\Gamma^{2}(\alpha)}{1+(2\alpha-\beta)\Gamma^{2}(\alpha)},
$$
and the extremal has the form
$$
x(t)=\frac{(2\alpha-\beta)\Gamma^{2}(\alpha)}{1+(2\alpha-\beta)\Gamma^{2}(\alpha)}
\left[1+\frac{1}{\Gamma^{2}(\alpha)}\int\limits_{0}^{t}(t-\tau)^{\alpha-1}(1-\tau)^{\alpha-\beta}d\tau
\right], \,\,\, t \in [0, \ 1].
$$

Now consider the case $\beta>\alpha>0.$ In this case by condition
(\ref{equ23}) we have
\begin{eqnarray*}
&(^{c}D_{0+}^{\alpha}x)(t)=0,\,\,\,t \in [0, \ 1],\nonumber\\
&x(0)=0, \ \ x(1)=1.
\end{eqnarray*}

Hence we get $x(t)=x(0)=0$, $t \in [0, 1]$. Therefore, problem
(\ref{equ30}) does not have solution.

We will consider problem (P$_{1}$). For any (fixed) variation $h
\in PC^{\alpha}([t_{0}, t_{1}], \ \mathbb{R}^{n}),$ $h(t_{1})=0$
the second variation of the functional satisfies the condition
$$
\delta^{2}J(x^{0}(\cdot), h(\cdot))=\frac{d^{2}}{d
\lambda^{2}}J(x^{0}+\lambda
h)|_{\lambda=0}
$$
$$
=\int\limits_{t_{0}}^{t_{1}}(t_{1}-t)^{\beta-1}
[\langle
P(t)(^{c}D_{t_{0}+}^{\alpha}h)(t),(^{c}D_{t_{0}+}^{\alpha}h)(t)\rangle
$$
\begin{equation}\label{equ32}
+2\langle Q(t)h(t), (^{c}D_{t_{0}+}^{\alpha}h)(t)\rangle+\langle
R(t)h(t), \ h(t)\rangle]dt+\langle l_{x_{0}x_{0}}h(t_{0}), \
h(t_{0})\rangle \geq 0. 
\end{equation}%
\end{example}
\begin{theorem}\label{teo44}
Let $0< \alpha \leq 1$, $\beta>0$ and the
function $x^{0}\in PC^{\alpha}([t_{0}, \,t_{1}],
\,\mathbb{R}^{n})$ deliver a weak local minimum in the problem
(P$_{1}$) ($x^{0}\in \texttt{wloc}\min P_{1}$), the functions $L,
L_{x}, L_{y}, L_{xx}, L_{xy}$ and $L_{yy}$ are continuous in some
neighborhood of the graph $\{(t, x^{0}(t),
(^{c}D_{t_{0}+}^{\alpha}x^{0})(t)| [t_{0}, t_{1}]\}$, the function
$l$ is twice continuously differentiable in the neighborhood of
the point $x^{0}(t_{0})$. Then the following inequality hold:
\begin{equation}\label{equ33}
\langle P(t)r, r \rangle \geq 0, \,\,\, \forall t \in T_{1},
\,\,\,  \forall r \in R^{n}.
\end{equation}%
\end{theorem}
\begin{proof}
To prove the theorem, we will argue by
contradiction. Let us assume that there exist a point $\sigma \in
T_{1}$ and a vector $r \in \mathbb{R}^{n}$ for which the
inequality
\begin{equation}\label{equ34}
\langle P(\sigma)r, r \rangle < 0
\end{equation}%
holds. Then exists a sufficiently small number $0<\varepsilon \ll
1$ and number $\gamma >0$ such that
\begin{equation}\label{equ35}
\langle P(t)r, r \rangle < -\gamma < 0, \,\,\, t \in
[\sigma-\varepsilon, \sigma+\varepsilon]\subset(t_{0}, t_{1}).
\end{equation}%

Define the function
$$
h(t)=h(t_{0})+\frac{1}{\Gamma(\alpha)}\int\limits_{t_{0}}^{t}(t-\tau)^{\alpha-1}g(\tau)d\tau,
$$
where
\[g(t)=\left\{ \begin{aligned}
r, \ \ \ \ \ \ \ \ \ \ \ \ \  t \in [\sigma-\varepsilon, \sigma+\varepsilon], \\
\\
0, \ \ \ \ t \in [t_{0}, t_{1}]\backslash [\sigma-\varepsilon,
\sigma+\varepsilon].
\end{aligned} \right.\]

We require that the condition $h(t_{1})=0$ be satisfied. Then
$$
h(t_{0})=-\frac{1}{\Gamma(\alpha)}\int\limits_{t_{0}}^{t_{1}}(t_{1}-\tau)^{\alpha-1}g(\tau)d\tau
=-\frac{r}{\Gamma(\alpha)}\int\limits_{\sigma-\varepsilon}^{\sigma+\varepsilon}(t_{1}-\tau)^{\alpha-1}d\tau
$$
$$
=-\frac{r}{\Gamma(\alpha+1)}[(t_{1}-(\sigma-\varepsilon))^{\alpha}-(t_{1}-(\sigma+\varepsilon))^{\alpha}].
$$

Obviously, when $t \in [t_{0}, \sigma-\varepsilon],$ we have
$$
h(t)=-\frac{r}{\Gamma(\alpha+1)}[(t_{1}-(\sigma-\varepsilon))^{\alpha}-(t_{1}-(\sigma+\varepsilon))^{\alpha}],
$$
when $t \in [\sigma-\varepsilon, \ \sigma+\varepsilon],$ we have
$$
h(t)=-\frac{r}{\Gamma(\alpha+1)}[(t_{1}-(\sigma-\varepsilon))^{\alpha}-(t_{1}-(\sigma+\varepsilon))^{\alpha}]
+\frac{r}{\Gamma(\alpha+1)}(t-(\sigma-\varepsilon))^{\alpha},
$$
and when $[\sigma+\varepsilon, t_{1}],$ we have
$$
h(t)=-\frac{r}{\Gamma(\alpha+1)}[(t_{1}-(\sigma-\varepsilon))^{\alpha}-(t_{1}-(\sigma+\varepsilon))^{\alpha}]
$$
$$
+\frac{r}{\Gamma(\alpha+1)}[(t-(\sigma-\varepsilon))^{\alpha}-(t-(\sigma+\varepsilon))^{\alpha}].
$$

Now, as before, we estimate each term included in formula (\ref{equ32}) of the second variation of the functional, the number of which is equal to four. In what follows, considering that we write the third term as the sum of three terms, the number $\Omega_k, k=5,...10$ expressing the absolute value of the $k$-th term in (\ref{equ32}) is six. 

Thus, using inequality (\ref{equ35}), the first term can be estimated as follows
\begin{eqnarray*}
&\int\limits_{t_{0}}^{t_{1}}(t_{1}-t)^{\beta-1}\langle
P(t)(^{c}D_{t_{0}+}^{\alpha}h)(t),
(^{c}D_{t_{0}+}^{\alpha}h)(t)\rangle
dt=\int\limits_{\sigma-\varepsilon}^{\sigma+\varepsilon}(t_{1}-t)^{\beta-1}\langle
P(t)r, r\rangle dt \\
&< \left\{ \begin{aligned}
-\gamma(t_{1}-t_{0})^{\beta-1}2\varepsilon,  \ \  0<\beta \leq 1, \ \ \ \      \\
\\
-\gamma(t_{1}-\xi)^{\beta-1}2\varepsilon, \ \ \ \ \beta >1,\ \ \ \
\ \ \ \
\end{aligned} \right. 
\end{eqnarray*}
whence
\begin{equation}\label{equ36}
\Omega_5 \leq -M_{5}\gamma \varepsilon,\; M_{5}=\min \{2(t_{1}-t_{0})^{\beta-1},
2(t_{1}-\xi)^{\beta-1}\},\;\sigma+\varepsilon < \xi < t_{1}.
\end{equation}%

Further, the second term can be bounded as follows
$$
\left|2\int\limits_{t_{0}}^{t_{1}}(t_{1}-t)^{\beta-1}\langle
Q(t)h(t),(^{c}D_{t_{0}+}^{\alpha}h)(t)\rangle dt \right|
$$
$$
=\Big|2\int\limits_{\sigma-\varepsilon}^{\sigma+\varepsilon}(t_{1}-t)^{\beta-1}
[(t-(\sigma-\varepsilon))^{\alpha}-((t_{1}-(\sigma-\varepsilon))^{\alpha}
$$
$$
-(t_{1}-(\sigma+\varepsilon))^{\alpha}]
\langle Q(t)r, r \rangle dt \Big|
$$
$$
\leq 2 \|Q\|
\|r\|^{2}\int\limits_{\sigma-\varepsilon}^{\sigma+\varepsilon}(t_{1}-t)^{\beta-1}(t-(\sigma-\varepsilon))^{\alpha}dt
+2\|Q\|
\|r\|^{2}((t_{1}-(\sigma-\varepsilon))^\alpha
$$
$$
-(t_{1}-(\sigma+\varepsilon))^{\alpha})
\int\limits_{\sigma-\varepsilon}^{\sigma+\varepsilon}(t_{1}-t)^{\beta-1}dt
$$
\[\leq \left\{ \begin{aligned}
K_{4}(t_{1}-\xi)^{\beta-1}\varepsilon^{1+\alpha},  \ \  0<\beta \leq 1, \ \ \ \      \\
K_{4}(t_{1}-t_{0})^{\beta-1}\varepsilon^{1+\alpha}, \ \ \ \ \beta
>1, \ \ \ \ \ \ \ \
\end{aligned} \right. \]
where $K_{4}=2^{\alpha+2}\frac{\alpha+2}{\alpha+1}\|Q\|\|r\|^{2}$. It means that 
\begin{equation}\label{equ37}
\Omega_6\leq M_{6}\varepsilon^{1+\alpha},\; M_{6}=\max \{K_{4}(t_{1}-\xi)^{\beta-1},
K_{4}(t_{1}-t_{0})^{\beta-1}\}.
\end{equation}%

As noted above, we write the third term as the sum of three terms. We
evaluate the  this terms as follows.

a) $$
\left|\int\limits_{t_{0}}^{\sigma-\varepsilon}(t_{1}-t)^{\beta-1}
\langle R(t)h(t), h(t) \rangle dt \right|
$$
$$
=\left|\frac{[(t_{1}-(\sigma-\varepsilon))^{\alpha}-(t_{1}-(\sigma+\varepsilon))^{\alpha}]^{2}}{\Gamma^{2}(\alpha+1)}
\int\limits_{t_{0}}^{\sigma-\varepsilon}(t_{1}-t)^{\beta-1}\langle
R(t)r, r \rangle dt \right|
$$
\[\leq \left\{ \begin{aligned}
K_{5}(t_{1}-\xi)^{\beta-1}(t_{1}-t_{0})\varepsilon^{\alpha+1},  \ \  0<\beta \leq 1, \ \ \ \      \\
\\
K_{5}(t_{1}-t_{0})^{\beta}\varepsilon^{\alpha+1}, \ \ \ \ \beta
>1, \ \ \ \ \ \ \ \ \ \ \ \ \ \ \ \ \ \ \ \ \
\end{aligned} \right.\] 

where $K_{5}=\frac{2^{\alpha+1} \alpha \|R\|
\|r\|^{2}}{\Gamma^{2}(\alpha+1)}(t_{1}-\xi)^{\alpha-1}$. Therefore
\begin{equation}\label{equ38}
\Omega_7\leq M_{7}\varepsilon^{\alpha+1},\; M_{7}=\max \{K_{5}(t_{1}-\xi)^{\beta-1}(t_{1}-t_{0}),
K_{5}(t_{1}-t_{0})^{\beta}\}.
\end{equation}%

b) $$
\left|\int\limits_{\sigma-\varepsilon}^{\sigma+\varepsilon}(t_{1}-t)^{\beta-1}
\langle R(t)h(t), h(t) \rangle dt \right| \leq
\frac{2\|R\|\|r\|^{2}}{\Gamma^{2}(\alpha+1)}
\int\limits_{\sigma-\varepsilon}^{\sigma+\varepsilon}(t_{1}-t)^{\beta-1}(t-(\sigma-\varepsilon))^{2\alpha}dt
$$
$$
+[(t_{1}-(\sigma-\varepsilon))^{\alpha}-(t_{1}-(\sigma+\varepsilon))^{\alpha}]^{2}
\int\limits_{\sigma-\varepsilon}^{\sigma+\varepsilon}(t_{1}-t)^{\beta-1}dt
$$
\[\leq \left\{ \begin{aligned}
K_{6}(t_{1}-\xi)^{\beta-1}\varepsilon^{2\alpha+1},  \ \  0<\beta \leq 1, \ \ \ \      \\
\\
K_{6}(t_{1}-t_{0})^{\beta-1}\varepsilon^{2\alpha+1}, \ \ \ \ \beta
>1,\ \ \ \ \ \ \ \
\end{aligned} \right. \]
where $K_{6}=\frac{2^{2\alpha+2} \|R\|
\|r\|^{2}}{\Gamma^{2}(\alpha+1)}\left(\frac{1}{2\alpha+1}+2\alpha(t_{1}-\xi)^{\alpha-1}\right)$. Finally we have
\begin{equation}\label{equ39}
\Omega_8\leq M_{8}\varepsilon^{2\alpha+1},\; M_{8}=\max \{K_{6}(t_{1}-\xi)^{\beta-1}, K_{6}(t_{1}-t_{0})^{\beta-1}\}.
\end{equation}%

c) $$
\left|\int\limits_{\sigma+\varepsilon}^{t_{1}}(t_{1}-t)^{\beta-1}
\langle R(t)h(t), h(t) \rangle dt \right| \leq
\frac{2\|R\|\|r\|^{2}}{\Gamma^{2}(\alpha+1)}
[((t_{1}-(\sigma-\varepsilon))^{\alpha}-(t_{1}-(\sigma+\varepsilon))^{\alpha})]^{2}
$$
$$
\times
\int\limits_{\sigma+\varepsilon}^{t_{1}}(t_{1}-t)^{\beta-1}dt+\int\limits_{\sigma+\varepsilon}^{t_{1}}(t_{1}-t)^{\beta-1}
((t-(\sigma-\varepsilon))^{\alpha}-(t-(\sigma+\varepsilon))^{\alpha})^{2}dt]\leq
\frac{2\|R\|\|r\|^{2}}{\Gamma^{2}(\alpha+1)}
$$
$$
\times \Big[\frac{\alpha}{\beta}(2\varepsilon)^{1+\alpha}
(t_{1}-(\sigma+\varepsilon))^{\alpha-1}(t_{1}-(\sigma+\varepsilon))^{\beta}
$$
$$
+\alpha(2\varepsilon)^{1+\alpha}
\int\limits_{\sigma+\varepsilon}^{t_{1}}(t_{1}-t)^{\beta-1}
(t-(\sigma+\varepsilon))^{\alpha-1}dt\Big]
$$
\[\leq \left\{ \begin{aligned}
K_{7}(t_{1}-\xi)^{\alpha+\beta-1}\varepsilon^{1+\alpha},  \ \  0<\alpha+\beta-1 \leq 1, \ \ \ \      \\
K_{7}(t_{1}-t_{0})^{\alpha+\beta-1}\varepsilon^{1+\alpha}, \ \ \ \
\alpha+\beta-1>1,\ \ \ \ \ \ \ \
\end{aligned} \right. \]
where $K_{7}=\frac{2^{\alpha+2}\alpha \|R\|
\|r\|^{2}}{\Gamma^{2}(\alpha+1)}\left(\frac{1}{\beta}+\frac{\Gamma(\beta)\Gamma(\alpha)}{\Gamma(\alpha+\beta)}\right)$. As a result, we get
\begin{equation}\label{equ40}
\Omega_9\leq M_{9}\varepsilon^{\alpha+1},\; M_{9}=\max \{K_{7}(t_{1}-\xi)^{\alpha+\beta-1},
K_{7}(t_{1}-t_{0})^{\alpha+\beta-1}\}.
\end{equation}%

The fourth term in (\ref{equ32}) can be bounded as follows
\begin{eqnarray}\label{equ41}
&|\langle l_{x_{0}x_{0}}h(t_{0}),
h(t_{0})\rangle|=\frac{1}{\Gamma^{2}(\alpha+1)}
[(t_{1}-(\sigma-\varepsilon))^{\alpha}\nonumber\\
&-(t_{1}-(\sigma+\varepsilon))^{\alpha}]^{2}
\langle l_{x_{0}x_{0}}r, r \rangle \leq
M_{10}\varepsilon^{\alpha+1}, 
\end{eqnarray}%
where  $M_{10}=\frac{2^{\alpha+1}\alpha \|l_{x_{0}x_{0}}\|
\|r\|^{2}}{\Gamma^{2}(\alpha+1)}(t_{1}-\xi)^{\alpha-1}.$

Taking into account estimates (\ref{equ36})-(\ref{equ41}) in (\ref{equ32}), we arrive at the
inequality
$$
\delta^{2}J(x(\cdot), h(\cdot))\leq \varepsilon
[-M_{5}\gamma+\varepsilon^{\alpha}(M_{6}+M_{7}+M_{8}\varepsilon^{\alpha}+M_{9}+M_{10})],
$$
which for sufficiently small $\varepsilon >0$ takes the form
$\delta^{2}J(x^{0}(\cdot), h(\cdot))< 0,$ which contradicts the
definition of optimality $\delta^{2}J(x^{0}(\cdot), h(\cdot))\geq
0.$ \qquad\qquad\qquad\qquad$\square$
\end{proof}

\subsection {Fractional Bolza problem }

For $0< \alpha \leq 1$ and $\beta>0$, the fractional Bolza problem
is called the following extremal problem without restrictions in
space $C^{\alpha}([t_{0}, t_{1}], \mathbb{R}^{n}):$
$$
(PB)\;\;B(x(\cdot))=\int\limits_{t_{0}}^{t_{1}}(t_{1}-t)^{\beta-1}L(t,
x(t),(^{c}D_{t_{0}+}^{\alpha}x)(t))dt+l(x(t_{0}),
x(t_{1}))\rightarrow extr. 
$$
The functional $(PB)$ is called the fractional Bolza functional,
the function $l$ is the terminant. Any functions of class
$C^{\alpha}([t_{0}, t_{1}], \mathbb{R}^{n})$ are admissible in
problem $(PB)$.
\begin{definition}\label{def42}
We will say that an admissible function
$x^{0}$ deliver a weak local minimum in problem $(PB)$, and write
$x^{0} \in wlocmin (PB)$ if there exists $\delta>0$ such that
$B(x(\cdot))\geq B(x^{0}(\cdot))$ for any admissible function $x$
for which
$$
\|x(\cdot)-x^{0}(\cdot)\|_{C^{\alpha}([t_{0},t_{1}],
\mathbb{R}^{n})}<\delta.
$$
\end{definition}
\begin{theorem}\label{teo45}
Let $0< \alpha \leq 1,$ $\beta>0$ and the
function $x^{0}$ deliver a weak local minimum in the problem (PB)
$(x^{0} \in wlocmin (PB)),$ the functions $L, \ L_{x}, \ L_{y}$ are
continuous in some neighborhood of the graph $\{(t, x^{0}(t),
(^{c}D_{t_{0}+}^{\alpha}x^{0})(t))| t \in [t_{0}, t_{1}]\},$ the
function $l$ is continuously differentiable in the neighborhood of
the point $(x^{0}(t_{0}), x^{0}(t_{1})).$ Then:

1) If $0< \alpha < \beta$, then the function $x^{0}$ is a solution
to the problem
\begin{eqnarray}\label{equ42}
(t_1-t)^{1-\alpha }&&(I_{t_{1}-}^{\alpha}b)(t)
+(t_{1}-t)^{\beta-\alpha}L_{y}(t)=0, \nonumber \\
&&\int\limits_{t_{0}}^{t_{1}}(t_{1}-t)^{\beta-1}L_{x}(t)dt+l_{x_{0}}=0,\\
&&l_{x_{1}}=0.\nonumber 
\end{eqnarray}
2) If $0< \beta \leq \alpha \leq 1,$ then the function $x^{0}$ is
a solution to the problem
\begin{eqnarray}\label{equ43}
&(t_1-t)^{1-\beta
}(I_{t_{1}-}^{\alpha}b)(t)+L_y\left(t\right)+l_{x_{1}}\frac{1}{\Gamma
\left(\alpha \right)} {\left(t_1-t\right)}^{\alpha -\beta }=0, \ \
t\in \left[t_0,\ t_1\right],\nonumber\\
&\int\limits^{t_1}_{t_0}{{\left(t_1-t\right)}^{\beta
-1}L_x\left(t\right)dt}+l_{x{_0}}+l_{x_{1}}=0.
\end{eqnarray}
\end{theorem}
Theorem \ref{teo45} is proved similarly to Theorem \ref{theo43}.

\begin{example}\label{exam43} For $0<\beta \leq \alpha \leq 1$ consider
the following Bolza problem in space $C^{\alpha}([0, 1],
\mathbb{R}):$
$$
B(x(\cdot))=\int\limits_{0}^{1}(1-t)^{\beta-1}(^{c}D_{0+}^{\alpha}x)^{2}(t)dt
+x^{2}(0)+x^{2}(1)\rightarrow extr.
$$

It is obvious that $L_{y}=2(^{c}D_{0+}^{\alpha}x),$
$l_{x_{0}}=2x(0),$ $l_{x_{1}}=2x(1).$ For this example, equation and condition in (\ref{equ43}) take the following form
\begin{equation}\label{equ48}
(^{c}D_{0+}^{\alpha}x)(t)=-\frac{x(1)}{\Gamma(\alpha)}(1-t)^{\alpha-\beta},
\,\,\, t \in [0, \ 1],
\end{equation}%
and
\begin{equation}\label{equ49}
x(0)+x(1)=0, 
\end{equation}%
respectively.

From (\ref{equ48}) we have
$$
x(t)=x(0)-\frac{x(1)}{\Gamma^{2}(\alpha)}\int\limits_{0}^{t}(t-\tau)^{\alpha-1}(1-\tau)^{\alpha-\beta}d\tau.
$$
Hence, under the condition (\ref{equ49}) we get
\begin{equation}\label{equ50}
x(t)=x(0)\left(1+\frac{1}{\Gamma^{2}(\alpha)}\int\limits_{0}^{t}(t-\tau)^{\alpha-1}(1-\tau)^{\alpha-\beta}d\tau
\right). 
\end{equation}%

From equalities (\ref{equ49}) and (46) we get
$$
x(0)\left(2+\frac{1}{(2\alpha-\beta)\Gamma^{2}(\alpha)}\right)=0.
$$

It follows that $x(0)=0.$ Therefore the extremal has the form
$x(t)=0,$ $t \in [0, 1].$
\end{example}
\begin{example}\label{exam44} For $0< \alpha< \beta$ consider the following
Bolza problem in space $C^{\alpha}([0, \ 1], \mathbb{R}):$
$$
B(x(\cdot))=\int\limits_{0}^{1}
(1-t)^{\beta-1}[x(t)+(^{c}D_{0+}^{\alpha}x)^{2}(t)]dt+x^{2}(0)
\rightarrow extr.
$$

It is obvious that $L_{x}=1$, $L_{y}=2y,$ $l_{x_{0}}=2x(0),$
$l_{x_{1}}=0.$ In this case, for the first equation and the second condition in (\ref{equ42}) we have
$$
(^{c}D_{0+}^{\alpha}x)(t)=-\frac{\Gamma(\beta)}{2\Gamma(\alpha+\beta)}(1-t)^{\alpha},
\,\,\, t \in [0, \ 1],
$$
and
$$
x(0)=-\frac{1}{2\beta}.
$$

Then for the stated problem the extremal has the form
$$
x(t)=-\frac{1}{2\beta}-\frac{\Gamma(\beta)}{2\Gamma(\alpha)\Gamma(\alpha+\beta)}
\int\limits_{0}^{t}(t-\tau)^{\alpha-1}(1-\tau)^{\alpha}d\tau,
\,\,\, t \in [0, \ 1].
$$
\end{example}
\begin{remark}\label{rem41} Let us consider the one-dimensional $(n=1)$
Bolza functional given by
$$
B(x(\cdot))=\int\limits_{0}^{1}(1-t)^{\beta-1}(^{c}D_{0+}^{\alpha}x)^{2}(t)dt
+x^{2}(0)\rightarrow extr.
$$
for all $x \in C^{\alpha}([0, \ 1], \mathbb{R}),$ where $0< \alpha
\leq 1$ and $\beta> \alpha >0$.

For this example, problem (\ref{equ42}) takes the following form
 \begin{eqnarray*}
(t_{1}-t)^{\beta-\alpha}&&(^{c}D_{0+}^{\alpha}x)(t)=0,\ \ \ \ \ t \in [t_{0}, t_{1}],  \\
&&x(0)=0.
\end{eqnarray*} 

It follows that the extremal has the form $x(t)=0,$ $t \in [0, \
1].$ This is optimal.

Now let's apply the result of work (see, Theorem 3.1. in \cite{10}) to
this example. For this example the Euler-Lagrange equations are
$$
D_{1-}^{\alpha}\left[\frac{(1-\cdot)^{\beta-1}}{\Gamma(\beta)}(^{c}D_{0+}^{\alpha}x)(\cdot)\right](t)=0.
$$
From here we have
$$
(^{c}D_{0+}^{\alpha}x)(t)=k(1-t)^{\alpha-\beta}\frac{\Gamma(\beta)}{2\Gamma(\alpha)},
$$
where $k$ is some constant. From this equality it follows that
$^{c}D_{0+}^{\alpha}x \notin L_{\infty},$ and this means that the
Euler-Lagrange equation obtained in work \cite{10} has no solution in
the space $_{c}AC_{0+}^{\alpha,\infty}$ (see \cite{10}). Therefore, in
space $C^{\alpha}([0, \ 1], \mathbb{R})$ has no solution.
\end{remark}
\begin{remark}\label{rem42} If we consider the case $\alpha=\beta=1$, that
is, the classical case, then problem (\ref{equ43}) has the form
\begin{eqnarray}\label{equ51}
&\int\limits^{t_1}_t{L_x(\tau )d\tau }+L_y(t)+l_{x{_1}}=0,\nonumber \\
&\int\limits^{t_1}_{t_0}L_x(t)dt+l_{x_{0}}+l_{x_{1}} =0.
\end{eqnarray}

From here it follows directly that the problem (\ref{equ51}) is
equivalent to the following Euler-Lagrange equation
$$
-\frac{d}{dt}L_{y}+L_{x}=0,
$$
with transversality conditions
$$
L_{y}(t_{0})=l_{x_{0}},
$$
$$
L_{y}(t_{1})=-l_{x_{1}}.
$$

Thus, we have seen that in classical calculus of variations the
Bolza problem can only be solved using the Du Bois-Reymond lemma,
without resorting to integration by parts.

Note that in article \cite{10} the Legendre conditions for problem
$(PB)$ were proved.
\end{remark}

\subsection {Problem with fixed initial condition and free final
condition}

Let $0< \alpha \leq 1$ and $\beta >0.$ Let's consider the
following problem in space $C^{\alpha}([t_{0}, t_{1}],
\mathbb{R}^{n}):$
\begin{eqnarray*}
(P_{2})\;\;&J(x(\cdot))=\int\limits_{t_{0}}^{t_{1}}(t_{1}-t)^{\beta-1}L(t,
x(t), (^{c}D_{t_{0}+}^{\alpha}x)(t))dt +l(x(t_{1}))\rightarrow
extr,\nonumber\\
&  x(t_{0})=x_{0}.
\end{eqnarray*}%

Here the segment $[t_{0}, t_{1}]$ is assumed to be fixed and
finite, $t_{0}< t_{1}$. $L(t, x, y)$ is a continuous function with
continuous partial derivatives $L_{x}$ and $L_{y}$, $l$ is a
continuous function with continuous derivative $l_{x_1}$.
\begin{theorem}\label{theo46} Let $0< \alpha \leq 1$, $\beta> 0$ and the
function $x^{0}$ deliver a weak local minimum in the problem
(P$_{2}$) $(x^{0} \in wlocmin P_{2})$, the functions $L, L_{x},
L_{y}$ are continuous in some neighborhood of the graph $\{(t,
x^{0}(t), (^{c}D_{t_{0}+}^{\alpha}x^{0})(t))| t \in [t_{0},
t_{1}]\},$ the function $l$ is continuously differentiable in the
neighborhood of the point $x^{0}(t_{1}).$ Then

1) if $0<\alpha < \beta$, then the function $x^{0}$ is a solution
to the problem
\begin{eqnarray*}
&(t_{1}-t)^{1-\alpha}(I_{t_{1}-}^{\alpha}b)(t)+(t_{1}-t)^{\beta-\alpha}L_{y}(t)=0,  \ \ \ \ t\in [t_{0}, t_{1}], \\
&x(t_{0})=x_{0},\,\,\,\,\, l_{x_{1}}=0,
\end{eqnarray*} 

2) if $0<\beta \leq \alpha \leq 1$, then the function $x^{0}$ is a
solution to the problem
\begin{eqnarray*}
&(t_{1}-t)^{1-\beta}(I_{t_{1}-}^{\alpha}b)(t)+L_{y}(t)+l_{x_{1}}\frac{(t_{1}-t)^{\alpha-\beta}}{\Gamma(\alpha)}=0,  \ \ \ t\in [t_{0}, t_{1}], \\
&x(t_{0})=x_{0}. 
\end{eqnarray*} 
\end{theorem}
Theorem \ref{theo46} is proved similary to Theorem \ref{theo43}.

The Legendre conditions for problem (P$_{2}$)are proved in papers
\cite{18,33,32}.

\section{ Conclusion}\label{sec5}

In this paper, we study the problems of minimizing a functional
depending on the Caputo fractional derivative of order $0<\alpha
\leq 1$ and the Riemann-Lioville fractional integral of order
$\beta >0$. We have generalized one of the most important lemmas
of the calculus of variations, the fundamental Du Bois-Reymond
lemma, to the fractional case. From this lemma it is clear that the dependence between the
parameters $\alpha$ and $\beta$ plays a significant role. The new
lemma allowed us to prove the fractional Euler-Lagrange equation
in integral form. Note that here we also take into account the
relationship between the parameters $\alpha$ and $\beta$, which
has not taken into account in previous works in the literature.
The usefulness of the obtained conditions is illustrated by
examples. Unlike a number of papers in the literature, here we
show that the standard proof of the Legendre condition in the
classical case $\alpha=1$ can be adapted to the fractional case
$0<\alpha<1$ when dealing with final constraints. To prove the
Legendre condition, we had to introduce a special variation
depending on the parameter $k$. Note that the method presented
here can be used to derive optimality conditions for other
fractional variational calculus problems.
\\
\section*{Declarations}
{\bf Conflict of interest}  The author declares that he has no conflict of interest. No funds, grants, or other support
were received.

\end{document}